\documentclass[11pt, oneside]{amsart}
\usepackage[T1]{fontenc}
\usepackage{geometry}                
\usepackage{graphicx}
\usepackage{amssymb}
\usepackage{epstopdf}
\usepackage{comment}
\usepackage{color}
\usepackage[
  colorlinks,
 citecolor=magenta,
 linkcolor=SAEblue
]{hyperref}
\title{Quantitative John-Nirenberg inequalities at different scales}
\author{Javier C. Mart\'inez-Perales}
\address[Javier C. Mart\'inez-Perales]{BCAM - Basque Center for Applied Mathematics, Bilbao, Spain.
} \email{jmartinez@bcamath.org}

\author[E. Rela]{Ezequiel Rela}
\address[Ezequiel Rela]{Departamento de Matem\'atica,
Facultad de Ciencias Exactas y Naturales, Universidad de Buenos Aires, Ciudad Universitaria Pabell\'on I, Buenos Aires 1428 Capital Federal Argentina} \email{erela@dm.uba.ar}

\author[I.P. Rivera-R\'{\i}os]{Israel P. Rivera-R\'{\i}os}
\address[Israel P. Rivera-R\'{\i}os] {Instituto de Matem\'atica de Bah\'{\i}a Blanca (INMABB), Departamento de Matem\'atica, Universidad Nacional del Sur (UNS) - CONICET, Av. Alem 1253, Bah\'{\i}a Blanca, Argentina}
\email{israel.rivera@uns.edu.ar}

\theoremstyle{plain}
\newtheorem{lettertheorem}{Theorem}
  \theoremstyle{plain}
  
  \theoremstyle{plain}
\newtheorem{theorem}{Theorem}[section] 
  \theoremstyle{plain}
  \newtheorem{corollary}{Corollary}[section] 
  \theoremstyle{plain}
   
     \theoremstyle{plain}
      
  \theoremstyle{plain}
       
  \theoremstyle{plain}
    \newtheorem{definition}{Definition}[section]  
  \newtheorem{lemma}{Lemma}[section] 
  \theoremstyle{remark}
  \newtheorem{remark}{Remark}[section] 
  \newtheorem{example}{Example}[section]

\def\Xint#1{\mathchoice
{\XXint\displaystyle\textstyle{#1}}%
{\XXint\textstyle\scriptstyle{#1}}%
{\XXint\scriptstyle\scriptscriptstyle{#1}}%
{\XXint\scriptscriptstyle%
\scriptscriptstyle{#1}}%
\!\int}
\def\XXint#1#2#3{{\setbox0=\hbox{$#1{#2#3}{%
\int}$ }
\vcenter{\hbox{$#2#3$ }}\kern-.6\wd0}}

\def\dashint{\thinspace \Xint-}

\renewcommand{\d}{\mathrm{d}}
\newcommand{\BMO}{\mathrm{BMO}}
\newcommand{\loc}{\mathrm{loc}}
\newcommand{\essinf}{\mathrm{ess\, inf}}
\newcommand{\esssup}{\mathrm{ess\, sup}}

 \numberwithin{equation}{section}

\usepackage{color}
\definecolor{SAEblue}{rgb}{0, .62, .91}

\definecolor{JaviBlue}{rgb}{.80, .50, 1}

\allowdisplaybreaks

\begin{document}
 
 \begin{abstract}
We provide an abstract estimate of the form
\[
 \|f-f_{Q,\mu}\|_{X \left(Q,\frac{\d \mu}{Y(Q)}\right)}\leq c(\mu,Y)\psi(X)\|f\|_{\mathrm{BMO}(\d\mu)}
\]
for all cubes $Q$ in $\mathbb{R}^n$ and every function $f\in \mathrm{BMO}(\d\mu)$, where $\mu$ is a doubling measure in $\mathbb{R}^n$, $Y$ is some positive functional defined on cubes, $\|\cdot \|_{X \left(Q,\frac{\d w}{w(Q)}\right)}$ is a sufficiently good quasi-norm and $c(\mu,Y)$ and $\psi(X)$ are positive constants depending on $\mu$ and $Y$, and $X$, respectively. 
That abstract scheme allows us to recover the sharp estimate 
\[
 \|f-f_{Q,\mu}\|_{L^p \left(Q,\frac{\d \mu(x)}{\mu(Q)}\right)}\leq c(\mu)p\|f\|_{\mathrm{BMO}(\d\mu)}, \qquad p\geq1
\]
for every cube $Q$ and every $f\in \mathrm{BMO}(\d\mu)$, which is known to be equivalent to the John-Nirenberg inequality,  and also enables us to obtain quantitative counterparts when $L^p$ is replaced by suitable strong and weak Orlicz spaces and $L^{p(\cdot)}$ spaces.

Besides the aforementioned results we also generalize \cite[Theorem 1.2]{Ombrosi2019} to the setting of doubling measures and obtain a new characterization of Muckenhoupt's $A_\infty$ weights.
 \end{abstract}

\maketitle

\tableofcontents


\section{Introduction}

The celebrated John-Nirenberg inequality for $\mathrm{BMO} $ functions has been extensively studied by many authors in different situations since its appearance in the original work \cite{John1961} by F. John and L. Nirenberg. In the Euclidean space $\mathbb{R}^n$ endowed with a doubling measure $\mu$ (see \eqref{eq:doubling_condition}) this inequality reads  
\begin{equation}\label{eq:JN_ineq}
\mu\left(\left\{x\in Q:|f(x)-f_{Q,\mu}|>t\right\}\right)\leq c_1e^{-t/\left(c_2\|f\|_{\mathrm{BMO}(\d\mu)}\right)}\mu(Q),\qquad t>0,
\end{equation}
where $f_{Q,\mu}:=\int_Qf(x)\,\d\mu(x)/\mu(Q)$,
and it is satisfied for every cube $Q$ in $\mathbb{R}^n$ and every $\mathrm{BMO}(\d\mu)$ function $f$ with universal constants $c_1,c_2>1$. Here $\|f\|_{\mathrm{BMO}(\d\mu)}$ denotes the $\mathrm{BMO}(\d\mu)$ norm of $f$, which is defined as
\[
\|f\|_{\mathrm{BMO}(\d\mu)} :=\sup_{Q\in\mathcal{Q}} \frac{1}{\mu(Q)}\int_Q|f(x)-f_{Q,\mu}|\,\d\mu(x),
\]
where the supremum is taken over the class $\mathcal{Q}$ of all cubes $Q$ in $\mathbb{R}^n$. Finiteness of this constant defines the belongingness of $f$ to the class $\mathrm{BMO}(\d\mu)$ of functions with bounded mean oscillations. 
 Recall that a  measure is said to satisfy the doubling condition, if there are positive constants $c_\mu$\index{constant! doubling  $c_\mu$} and $n_\mu$\index{dimension! doubling $n_\mu$} such that, for every pair of cubes $Q$ and $\tilde{Q}$ in $\mathbb{R}^n$ with $Q\subset \tilde{Q}$, the inequality \begin{equation}\label{eq:doubling_condition}
\mu(\tilde{Q})\leq c_\mu \left(\frac{\ell(\tilde{Q})}{\ell(Q) }\right)^{n_\mu}\mu(Q)
\end{equation}
holds.  The constants $c_\mu$ and $n_\mu$ are called doubling constant and doubling dimension of $\mu$, respectively. Note that $c_\mu$ must be larger than $1$.

It is known that a John-Nirenberg inequality \begin{equation}\label{eq:abstract_JN_ineq}
\mu\left(\left\{x\in Q:|f(x)-f_{Q,\mu}|>t\right\}\right)\leq Ce^{-c(f)\cdot t }\mu(Q),\qquad t>0,\qquad Q\in\mathcal{Q},
\end{equation}
with constants $c(f),C>0$,  characterizes the belongingness of a function $f$ to the $\mathrm{BMO}(\d\mu)$ class, and in that case $C$ and $c(f)$ can be taken to be the  constants $c_1$ and  $(c_2\|f\|_{\mathrm{BMO}(\d\mu)})^{-1}$ in \eqref{eq:JN_ineq}. A John-Nirenberg inequality \eqref{eq:abstract_JN_ineq} for a locally integrable function $f$ is in turn equivalent to the validity of a precise estimate (actually, a family of estimates) of the form
\begin{equation}\label{eq:estimacion_PR}
\left(\frac{1}{\mu(Q)}\int_Q|f(x)-f_{Q,\mu}|^p\,\d\mu(x)\right)^{\frac{1}{p}}\leq c(\mu)\cdot p \cdot C(f)
\end{equation}
for all cubes $Q\in\mathcal{Q}$ and all $p>1$, with $c(\mu)>0$ and $C(f)>0$ independent of $p$ and $Q$. Moreover, the constant $C(f)$ in \eqref{eq:estimacion_PR} can be replaced by $\|f\|_{\mathrm{BMO}(\d\mu)}$. See \cite{Perez2019} and also \cite[p.146]{Stein1993} for this.  

It turns out then that having inequality \eqref{eq:estimacion_PR} for every $p>1$ becomes a precise quantitative expression of the John-Nirenberg inequality at all the $L^p$ scales and then we will give \eqref{eq:estimacion_PR} precisely the name of quantitative John-Nirenberg inequality at the $L^p$ scale.  The aim in this work is to get precise inequalities in the spirit of \eqref{eq:estimacion_PR}  to obtain  variants of the quantitative  John-Nirenberg inequality by replacing the $L^p$ norms by different norms. To be precise, the main topic of this paper is  the search for a method that allows to get  precise inequalities like \eqref{eq:estimacion_PR} for $\mathrm{BMO}(\d\mu)$ functions beyond the $L^p(\d w)$ scale. 
Here and in the remainder of this work, we denote by  $\d w$ to the measure given by $w(x)\d x$, where $w$ is a weight: a non negative locally integrable function in $\mathbb R^n$.
The natural approach to extend inequality \eqref{eq:estimacion_PR} is to study different functions spaces endowed with a notion of (at least) a quasi-norm allowing us to define a sort of local average. The precise definitions will be given 
in detail in Definition \ref{def:quasinorm}. Accepting for a moment that we do have such a notion of a quasi-norm different from the $L^p(\d w)$ norm,  denoted by $\|\cdot\|_{X(\d \nu)}$, our aim will be to find a method giving estimates of the form
\begin{equation}\label{question2}
 \|f-f_{Q,\mu}\|_{X \left(Q,\frac{\d \nu}{\nu(Q)}\right)}\leq c(\mu)\psi(X)\|f\|_{\mathrm{BMO}(\d\mu)}
\end{equation} 
for every cube $Q\in\mathcal{Q}$, 
 where $\|\cdot\|_{X \left(Q,\frac{\d \nu}{\nu(Q)}\right)}$ is the aforementioned notion of a local average defined in terms of the quasi-norm and  $\psi(X)$ is some precise constant depending on that quasi-norm.   We will go even further by considering modified  averaged measures of the form $\d \nu/Y(Q)$, where  $Y:\mathcal{Q}\to (0,\infty)$ is some functional defined over cubes.
 
Let us depict a possible and quite natural path for getting results of this type.  Take a function $\phi$ and suppose that the local Luxemburg type norm 
\begin{equation}\label{eq:Luxemburg}
\|f\|_{\phi(L)\left(Q,\frac{\d \mu}{\mu(Q)}\right)}:=\inf\left\{\lambda>0:\frac{1}{\mu(Q)}\int_Q\phi\left(\frac{|f(x)|}{\lambda}\right)\,\d\mu(x)\leq1 \right\},
\end{equation}
is well defined for every cube $Q$ in $\mathbb{R}^n$. If $\phi$ is an increasing function with $\phi(0)=0$ which is absolutely continuous on every compact interval of $[0,\infty)$, then we know by Fubini's theorem that the following so-called layer-cake representation formula holds:
\[
\int_{Q}\phi\left[|f(x)|\right]\,\d\mu(x)=\int_0^\infty \phi'(t)\mu\left(\left\{x\in Q:|f(x)|>t\right\}\right)\,\d t,
\]
for any cube $Q$ of $\mathbb{R}^n$ and any measurable function $f$. Let us suppose that  $f\in\mathrm{BMO}(\d\mu)$. We know then that $f$  satisfies the John-Nirenberg inequality \eqref{eq:JN_ineq} and so, for any $\lambda>0$,  
\[
\begin{split}
 \frac{1}{\mu(Q)}\int_Q\phi\left(\frac{|f(x)-f_{Q,\mu}|}{\lambda}\right)\,\d\mu(x)&= \frac{1}{\mu(Q)}\int_0^\infty \phi'(t)\mu\left(\left\{x\in Q: |f(x)-f_{Q,\mu}|>\lambda t \right\}\right)\,\d t\\
 & \leq c_1 \int_0^\infty \phi'(t)e^{-\lambda t/c_2\|f\|_{\mathrm{BMO}(\d\mu)}}\,\d t \\
 &= c_1\mathcal{L}\{\phi'\}\left( \lambda /c_2\|f\|_{\mathrm{BMO}(\d\mu)} \right),
\end{split}
\]
where $\mathcal{L}$ represents the Laplace transform. If in addition the function $\phi$ is convex, then one has that $\phi'$ is positive, which makes $\mathcal{L}\{\phi'\}$ a decreasing function on $(0,\infty)$. Therefore, we can invert it and so, we know that $c_1\mathcal{L}\{\phi'\}\left( \lambda/c_2\|f\|_{\mathrm{BMO}(\d\mu)}\right)\leq 1$ if and only if $\lambda\geq c_2 \|f\|_{\mathrm{BMO}(\d\mu)}\mathcal{L}\{\phi'\}^{-1}\left(\frac{1}{c_1}\right)$. Hence, for any function $f\in \mathrm{BMO}(\d\mu)$, and for a function $\phi$ as the one depicted above, we have that
\[
\|f-f_{Q,\mu}\|_{\phi(L)\left(Q,\frac{\d\mu}{\mu(Q)}\right)}\leq c_2\|f\|_{\mathrm{BMO}(\d\mu)}\mathcal{L}\{\phi'\}^{-1}\left(\frac{1}{c_1}\right).
\]
Moreover, given any doubling measure $\mu$, it is known the existence of a function $\tilde{f}\in \mathrm{BMO}(\d\mu)$ satisfying that
\[
\mu\left(\left\{x\in Q:|\tilde{f}(x)-\tilde{f}_{Q,\mu}|>t\right\}\right)\geq C(\mu)e^{-t/ c(\mu) }\mu(Q),\qquad t>0
\]
for any cube $Q$ in $\mathbb{R}^n$, where $C(\mu)$ and $c(\mu)$ are positive constants depending only on the underlying measure $\mu$. This proves that the exponential behaviour of the level sets in the John-Nirenberg inequality \eqref{eq:JN_ineq} is the best one can get in general for $\mathrm{BMO}(\d\mu)$ functions. It also says that the estimate 
\[
\|f-f_{Q,\mu}\|_{\phi(L)\left(Q,\frac{\d\mu}{\mu(Q)}\right)}\leq c_2\|f\|_{\mathrm{BMO}(\d\mu)}\mathcal{L}\{\phi'\}^{-1}\left(\frac{1}{c_1}\right)
\]
for every cube $Q$ in $\mathbb{R}^n$ is essentially optimal, since there is a function $\tilde{f}\in \mathrm{BMO}(\d\mu)$ and positive constants $C(\mu)$ and $c(\mu)$ such that
\[
\|\tilde{f}-\tilde{f }_{Q,\mu}\|_{\phi(L)\left(Q,\frac{\d\mu}{\mu(Q)}\right)}\geq  c(\mu)\|\tilde{f}\|_{\mathrm{BMO}(\d\mu)}\mathcal{L}\{\phi'\}^{-1}\left(\frac{1}{C(\mu)}\right)
\]
for every cube $Q$ in $\mathbb{R}^n$. 

This then provides a method for proving quantitative John-Nirenberg inequalities like \eqref{question2} with an optimal control in the constant $\psi(X)$ as far as the norm is given by a Luxemburg norm defined by a function $\phi$ like the one considered above.  Note that this approach gives an alternative proof of the sharp inequality
 \eqref{eq:estimacion_PR}, as the $L^p$ norm is a particular case of the Luxemburg norm given above if  we choose $\phi_p(t)=t^p$, and the quantity  $\mathcal{L}\{\phi_p'\}^{-1}\left(\frac{1}{C}\right)$   behaves aymptotically like $p$ when $p\to\infty$, for any $C>0$. However, although  it is easy to compute the inverse of the Laplace transform of   $\phi_p'$, it seems not to be the case for other functions $\phi$. Also, the method is  confined to the study of norms given by the Luxemburg method in terms of some special functions $\phi$, and this rules out interesting norms as for instance the ones of variable Lebesgue spaces. It is our purpose in this paper to give a general procedure which allows to prove a quantitative John-Nirenberg inequality like \eqref{question2} for more general norms without computing the inverse of a Laplace transform.

Our method is based in a generalization of the self-improving result \cite[Theorem 1.5]{Perez2019} in which a very simple method based on the Calder\'on-Zygmund decomposition is used (see also \cite[pp. 31--32]{Journe1983}, where the original ideas inspiring the general result can be found). The special case of \cite[Theorem 1.5]{Perez2019} which is of interest for us is the following.

\begin{lettertheorem}\label{thm:PerezRela}
Let $\mu$ be a doubling measure in $\mathbb{R}^n$. There exists a geometric constant $c(\mu)>0$ such that, given any $p\geq 1$, the inequality
\[
\left(\frac{1}{\mu(Q)}\int_Q |f(x)-f_{Q,\mu}|^p\,\d\mu(x)\right)^{1/p}\leq c(\mu)\cdot p\|f\|_{\mathrm{BMO}(\d\mu)}
\]
holds for any cube $Q$ in $\mathbb{R}^n$ and any function $f\in\mathrm{BMO}(\d\mu)$.
\end{lettertheorem}

The main contribution of our work is to provide analogous results for more general norms  at the left-hand side of the inequality in Theorem \ref{thm:PerezRela}. 
To that end, we will take further some ideas in \cite{Perez2019} and will also consider some of the concepts appearing in \cite{Ombrosi2019}, thus including in the theory more general $\mathrm{BMO}$ spaces defined by different oscillations. 

Now, to be able to present the main result, we need to describe the key concept for our purposes, namely the notion of \emph{local average}. It will be assumed through the rest of the work that such a concept is defined and denoted by $\|\cdot\|_{X\left(Q,\frac{\d\nu}{Y(Q)}\right)}$. This local average must be thought as a general extension of a standard local average as the one given in \eqref{eq:Luxemburg}. An example of this situation is given by the construction  $\|f\|_{X\left(Q,\frac{\d\nu}{Y(Q)}\right)}:=\|f\cdot\chi_Q\|_{X\left(\mathbb{R}^n,  \d\nu/Y(Q)\right)}$ for every cube $Q$, where we assume that $\|\cdot\|_{X\left(\mathbb{R}^n,\frac{\d\nu}{Y(Q)}\right)}$ is well defined for every cube $Q$ given the functional $Y:\mathcal{Q}\to(0,\infty)$.  This is for instance the case of function norms defined by a Luxemburg norm, and is the approach we took for our examples. Another possible choice  for the definition of a local average (which does not seem to allow the consideration of a functional $Y$) is $\|f\|_{X\left(Q,\frac{\d\nu}{\nu(Q)}\right)}:=\frac{\|f\cdot\chi_Q\|_{X\left(\mathbb{R}^n, \d\nu \right)}}{\|\chi_Q\|_{X\left(\mathbb{R}^n, \d\nu \right)}}$, which makes sense for any quasi-normed function space over the measure space $(\mathbb{R}^n,\d\nu)$.   This is the approach taken for instance in \cite{Ho2014}.

 As we already mentioned, we gain generality in our results by considering the functional $Y$ defined over the family of cubes in $\mathbb{R}^n$. Trivial examples are $Y(Q)=w(Q)$  or the functional $Y(Q)=w_r(Q)$ defined by 
 \[
w_r(Q):=\mu(Q)^{1/r'}\left(\int_Q w(x)^r\,\d\mu(x)\right)^{1/r}, \qquad r>1,
\]
 for a weight $w$ (see the discussion after \eqref{eq:SDp_modified} for more details about this functional $w_r$). We will impose two main conditions on this functional $Y$: certain compatibility with respect to the quasi-norm detailed in  Definition \ref{def:compatible} and a sort of $A_\infty$ type condition in the spirit of the $A_\infty$ condition for Muckenhoupt weights described in Definition \ref{def:generalized_Ainfty2}.

We present now our general theorem, that can be seen as a \emph{template} from which we can derive a series of particular cases of self-improving results for different classical function spaces.

\begin{theorem}\label{thm:self-improving_general}
Let $\mu$ be a doubling measure on $\mathbb{R}^n$ and consider a weight $w\in L^1_{\loc}(\mathbb{R}^n,\d\mu)$. Let $(X(w),\|\cdot\|_{X(w)})$ be a quasi-normed function space. Let $Y$ be an  $A_\infty(\d\mu,X(w))$ functional compatible with $X(w)$ (see Definition \ref{def:compatible}) with associated increasing bijection $\Psi$
(see Definition \ref{def:generalized_Ainfty2}).   Then there is a constant $C(\mu,\Psi)>0$ such that, for any $f\in \mathrm{BMO}(\d\mu)$ the following  holds 
\[
\left\|f-f_{Q,\mu}\right\|_{X\left(Q,\frac{\d w}{Y(Q)}\right)}\leq C\left(\mu,\Psi  \right) \|f\|_{\mathrm{BMO}(\d\mu)},\qquad Q\in\mathcal{Q}.
\]
Moreover, we can take
 \[
    C\left(\mu, \Psi \right) :=\inf_{L>\max\left\{1,\left[{\Psi} \left(({C_Y}\cdot K)^{-1} \right)\right]^{-1}\right\}}      c_\mu 2^{n_\mu}\frac{L }{ 1- {C_Y}\cdot K\cdot{\Psi}^{-1}\left(\frac{1}{L}\right) },
 \]
 where ${C_Y}$ is the constant in the $A_\infty$ condition for $Y$.
\end{theorem}

In such a generality, it could be not easy to grasp the reach of the theorem, but its power becomes clear in light of the large variety of particular examples that can be treated in a unified manner.

Two different explicit examples will be given. The first one provides a quantitative John-Nirenberg inequality like \eqref{question2} for Orlicz type norms $\|\cdot\|_{\phi(L)(w)}$ defined by submultiplicative Young functions. The application of this  approach to the specific norms $\|\cdot\|_{L^p\log^\alpha L(\d x)}$, $p\geq 1$, $\alpha\geq0$ is investigated.  In this case, the following result is obtained.
\begin{corollary}\label{cor:LlogL}
Let $\mu$ be a doubling measure in $\mathbb{R}^n$ and consider $p> 1$, $\alpha\geq0$. Then
\[
\begin{aligned}
 \|f-f_{Q,\mu}&\|_{L^p\log^\alpha L\left(Q,\frac{\d\mu}{\mu(Q)}\right)}\leq c_\mu2^{n_{\mu}}e2^{\alpha}\left(p+\alpha+1\right) \|f\|_{\mathrm{BMO}(\d\mu)}
 \end{aligned}
 \] 
 for every cube $Q$ in $\mathbb{R}^n$ and every function $f\in\mathrm{BMO}(\d\mu)$.
\end{corollary}

 Observe that this extends the classical case  in Theorem \ref{thm:PerezRela} to a wider collection of spaces, as the precise estimate for the $L^p$ case is obtained by taking $\alpha=0$. 
  
The second example which will be presented corresponds to the variable Lebesgue norms $\|\cdot\|_{L^{p(\cdot)}(\d x)}$, which shows that our method is more flexible than the one based on the use of the Laplace transform. The precise statement is the following
  (see Section \ref{sec:variable-Lp} for the precise details and definitions).
 \begin{corollary}\label{cor:variableLebesgue}
 Consider an essentially bounded exponent function $p:\mathbb{R}^n\to\mathbb{R}$ with essential  upper bound $p^+$. There exists a constant $C(n)>0$  such that
 \[
 \|f-f_Q\|_{L^{p(\cdot)}\left(Q,\frac{\d x}{|Q|}\right)}\leq C(n) p^+ \|f\|_{\mathrm{BMO}}
 \] 
 for every cube $Q$ in $\mathbb{R}^n$ and every function $f\in \mathrm{BMO}$.
 \end{corollary}
 From such an inequality, we deduce the following   generalized John-Nirenberg inequality.
\begin{corollary}\label{cor:JohnNirenbergvariable} 
Let $p:\mathbb{R}^n\to\mathbb{R}$ be an essentially bounded exponent function with essential upper bound $p^+$. There exists a constant $C(n,p^+)>0$ such that the John-Nirenberg type inequality 
     \[
\|\chi_{\{x\in Q: |f(x)-f_Q|\geq t\}}\|_{L^{p(\cdot)}\left(Q,\frac{\d x}{|Q|}\right)}\leq  2 e^{-C(n,p^+)t/\|f\|_{\mathrm{BMO}}}
\]
holds for every cube $Q$ in $\mathbb{R}^n$ and every function $f\in \mathrm{BMO}$.
\end{corollary} 
This John-Nirenberg type inequality is related to that obtained in \cite{Ho2014}, where a different $L^{p(\cdot)}$ average is considered. It is a remarkable fact that no further condition has to be imposed on $p$ to satisfy the above inequalities, in contrast with the result \cite{Ho2014} where besides the essential uniform boundedness, local log-H\"older conditions for the exponent function $p$ are imposed.

The type of techniques which are studied here are flexible enough to be applicable in many different situations. An example of this is the fact that new generalized Karagulyan type estimates can be obtained under suitable conditions for these quasi-norms. This is part of another ongoing project of the first author.

Along this work, we will write $A\lesssim B$ whenever there is some constant $C>0$, independent of the relevant parameters, such that $A\leq C\cdot B$. We will stress the dependence of some constant $C$ on a certain parameter $\alpha$ by including it in a parenthesis like this: $C(\alpha)$.  The notation  $A\gtrsim B$ will mean that $B\lesssim A$ and $A\asymp B$ will be used in case both $A\lesssim B $ and $A\gtrsim B$ hold at the same time.

The rest of the paper is organized as follows:  in Section \ref{sec:concepts}  we introduce some previous   self-improving results and we discuss their hypotheses. This leads us to consider a generalization of $A_\infty$ weights in relation to $L^p$ norms which we later extend to the context of general quasi-normed spaces.   In Section \ref{sec:self-improve} use the general $A_\infty$ condition to settle Theorem \ref{thm:self-improving_general}. Section \ref{sec:applications} is devoted to provide corollaries of Theorem \ref{thm:self-improving_general}, among which Corollaries \ref{cor:LlogL} and \ref{cor:variableLebesgue} are included. We include an appendix with the proof of Theorem \ref{pr:characterization_Ainfty}, which is a generalization to the setting of doubling measures of \cite[Theorem 1.2]{Ombrosi2019}.

\section{The \texorpdfstring{$A_\infty$}{Ainfty} condition of a functional with respect to a quasi-norm}\label{sec:concepts}
In this section we provide the fundamental tools and concepts used to prove the main results of the paper. Some aspects of the previous self-improving result \cite[Theorem 1.5]{Perez2019} and some results in \cite{Ombrosi2019} will be discussed to motivate one of the new concepts which will be introduced here, namely, the generalized $A_\infty$ type condition adapted to general quasi-normed function spaces. 
 The basic assumptions on the quasi-norms we will consider in this work will also be introduced here.  
The core ideas for the results in this paper come essentially from \cite[Theorem 1.5]{Perez2019}. We state this result here for the convenience of the reader. We first recall that a weight $w$ is an $A_\infty$ weight if there exist some $\delta,C>0$ such that, given a cube $Q$ in $\mathbb{R}^n$, the inequality
\begin{equation}\label{eq:(P3')}\frac{w(E)}{w(Q)}\leq C\left(\frac{\mu(E)}{\mu(Q)}\right)^\delta 
\end{equation}
holds for any measurable subset $E\subset Q$.
We also recall the standard notation $\Delta(Q)$ for the family  of  countable disjoint families of subcubes of a given cube $Q$.

\begin{lettertheorem}\label{thm:Theorem1.5}
Let $\mu$ be a doubling measure in $\mathbb{R}^n$ and consider $w\in A_\infty(\d\mu)$.
 Suppose that, for   a functional $a:\mathcal{Q}\to(0,\infty)$ there exist $p\geq 1$, $s>1$ and $\|a\|>0$ such that, for every cube $Q$ in $\mathbb{R}^n$, the inequality
\begin{equation}\label{eq:SDp}
\left(\sum_{j\in\mathbb{N}}\left(\frac{a(Q_j)}{a(Q)}\right)^p\frac{w(Q_j)}{w(Q)}\right)^{1/p}\leq \|a\|\left(\frac{\mu\left(\bigcup_{j\in\mathbb{N}} Q_j\right)}{\mu(Q)}\right)^{1/s}
\end{equation}
holds for any  $\{Q_j\}_{j\in\mathbb{N}}$ in $\Delta(Q)$. There exists a constant $C(\mu)>0$ such that, for every $f\in L^1_{\loc}(\mathbb{R}^n,\d\mu)$ with
\begin{equation}\label{eq:starting_poinr_PR}
\frac{1}{\mu(Q)}\int_Q|f(x)-f_{Q,\mu}|\,\d\mu(x)\leq a(Q),\qquad Q\in\mathcal{Q},
\end{equation}
the estimate 
\begin{equation}\label{eq:improvement_PR}
\left(\frac{1}{\mu(Q)}\int_Q|f(x)-f_{Q,\mu}|^p\,\d\mu(x)\right)^{1/p}\leq C(\mu)\,s\,\|a\|^s\, a(Q),
\end{equation}
holds for any cube $Q$ in $\mathbb{R}^n$.
\end{lettertheorem}
The proof of Theorem \ref{thm:Theorem1.5} is based in a Calder\'on-Zygmund decomposition which takes advantage of the two main hypothesis of the result, namely, the $SD_p^s(w)$ condition \eqref{eq:SDp} and the $A_\infty(\d\mu)$ condition on $w$.

As already observed in \cite[Remark 1.6]{Perez2019}, the $A_\infty(\d\mu)$ condition on $w$ seems to be an artifice of the proof and it may be not needed for getting the general result. The authors use the $A_\infty$ condition as a tool for proving that the auxiliary functional $a_\varepsilon(Q):=a(Q)+\varepsilon$ satisfies a smallness condition like \eqref{eq:SDp} provided that the original functional $a$ satisfies it. More specifically, they deal with the following computation for any cube $Q$ and any   $\{Q_j\}_{j\in\mathbb{N}}\in\Delta(Q)$:
 \[
 \left(\sum_{j\in\mathbb{N}}\frac{a_\varepsilon(Q_j)^pw(Q_j)}{a_\varepsilon(Q)^pw(Q)}\right)^{1/p}\leq   \left(\sum_{j\in\mathbb{N}}\frac{a (Q_j)^pw(Q_j)}{a (Q)^pw(Q)}\right)^{1/p}+ \left( \frac{ w\left(\bigcup_{j\in\mathbb{N}}Q_j\right)}{ w(Q)}\right)^{1/p}.
 \]
Note that the $A_\infty(\d\mu)$ condition \eqref{eq:(P3')} is what allows to bound the second term in the sum above to finally get a smallness condition like \eqref{eq:SDp} on $a_\varepsilon$ for any $\varepsilon>0$.

The need of this condition for a self-improving result like Theorem \ref{thm:Theorem1.5} has been investigated in \cite{Martinez2020}, where the  first author studies alternative arguments avoiding the $A_\infty(\d\mu)$ condition on the weight to get a self-improving like that. Although the results there are not fully satisfactory in the sense that they do not recover the improvement \eqref{eq:improvement_PR} without the $A_\infty(\d\mu)$ condition, they are good enough to get a new unified approach for getting classical and fractional weighted Poincar\'e-Sobolev inequalities. The approach taken there consists on replacing the weight $w$ in the $SD_p^s(w)$ condition \eqref{eq:SDp} by a slightly more general functional $w_r$ defined by 
\[
w_r(Q):=\mu(Q)^{1/r'}\left(\int_Q w(x)^r\,\d\mu(x)\right)^{1/r},
\]
thus getting a modified $SD_p^s(w)$ condition which reads as follows
\begin{equation}\label{eq:SDp_modified}
\left(\sum_{j\in\mathbb{N}}\left(\frac{a(Q_j)}{a(Q)}\right)^p\frac{w_r(Q_j)}{w_r(Q)}\right)^{1/p}\leq \|a\|\left(\frac{\mu\left(\bigcup_{j\in\mathbb{N}} Q_j\right)}{\mu(Q)}\right)^{1/s}
\end{equation}
for any cube $Q$ and any   $\{Q_j\}_{j\in\mathbb{N}}\in\Delta(Q)$.

Observe that $w_r(Q)$ is the result of applying  Jensen's inequality to the classical functional defined by $w(Q)$ for any cube $Q\in\mathcal{Q}$. This kind of functionals already appeared in some works as for instance   \cite{Perez1995,CruzUribe2000,CruzUribe2002}, in which the authors study sufficient conditions for the two-weighted weak and strong-type (respectively) boundedness of fractional integrals, Calder\'on-Zygmund operators and commutators. There, one can find the following straightforward properties of $w_r$: 
\begin{enumerate}
\item $w(E)\leq w_r(E)$ for any measurable nonzero measure set $E$.  
\item If $E\subset F$ are two nonzero measure sets, then 
\begin{equation}\label{w_r1}
w_r(E)\leq \left(\frac{\mu(E)}{\mu(F)}\right)^{1/r'}w_r(F).
\end{equation}
\item If $E=\bigcup_{j\in\mathbb{N}}E_j$ for some disjoint family $\{E_j\}_{j\in\mathbb{N}}$, then 
\begin{equation}\label{w_r2}
\sum_{j\in\mathbb{N}}w_r(E_j) \leq w_r(E).
\end{equation}
\item If two measurable sets $E$ and $F$ satisfy $E\subset F$, then 
\begin{equation}\label{w_r3}
w_r(E)\leq w_r(F).
\end{equation}
\end{enumerate}
The above properties are what  allow to prove a smallness condition for the perturbations $a_\varepsilon$ of a functional $a$. Specially, condition \eqref{w_r1} is what  makes possible to work with these perturbations  without assuming the $A_\infty(\d\mu)$ condition on the weight. 

Also related with this problem, and more related to the results which will be studied here is the work \cite{Ombrosi2019}, where the embedding of $\mathrm{BMO}(\d\mu)$ into certain weighted $\mathrm{BMO}$ spaces is characterized. To be precise, they consider the following weighted $\mathrm{BMO}$ spaces. 

\begin{definition}\label{def:BMO}
Let us consider a positive functional $Y:\mathcal{Q}\to(0,\infty)$  defined over the family $\mathcal{Q}$ of all cubes of $\mathbb{R}^n$. Consider a measure $\mu$ in $\mathbb{R}^n$ and pick a weight $v\in L^1_{\loc}(\mathbb{R}^n,\mu)$. We define the class of functions with bounded $(v\,\d\mu,Y)$-mean oscillations as
\begin{equation}\label{eq:BMO}
\BMO_{v\,\d\mu,Y}:=\left\{f\in L_{\loc}^1(\mathbb{R},\d\mu):\|f\|_{\BMO_{v\,\d\mu,Y}}<\infty\right\},
\end{equation}
where 
\begin{equation}\label{eq:BMO2}
\|f\|_{\BMO_{v\,\d\mu,Y}}:= \sup_{Q\in\mathcal{Q}}\frac{1}{Y(Q)}\int_Q|f(x)-f_{Q,\mu}|v(x)\,\d\mu(x). 
\end{equation}
For the special case $v=1$, $Y(Q)=\mu(Q)$ the notation $\BMO(\d\mu)$ will be adopted.
\end{definition}
It is one of the main results in \cite{Ombrosi2019} that the embedding inequality
\[
\|f\|_{\mathrm{BMO}_{v\,\d x,Y}}\leq B\|f\|_{\mathrm{BMO}(\d x)}
\]
is valid if and only if the weight $v$ and the functional $Y$ satisfy the Fujii-Wilson type $A_\infty$ condition 
 \begin{equation}\label{eq:Fujii-WilsonY}
 [v]_{A_{\infty,Y}}:=\sup_{Q\in\mathcal{Q}}\frac{1}{Y(Q)}\int_Q M(v\chi_Q)(x)\,\d x<\infty,
 \end{equation}
 which, in case $Y(Q)=|Q|$, coincides with the $A_\infty(\d x)$ condition introduced in \eqref{eq:(P3')} (see \cite{Duoandikoetxea2016}).
 The following theorem generalizes the aforementioned result in \cite{Ombrosi2019} to the setting of doubling measures.  A proof of it is provided in Appendix \ref{proof:characterization_Ainfty}.
\begin{theorem}\label{pr:characterization_Ainfty}
Let $\mu$ be a doubling measure in $\mathbb{R}^n$ and consider a functional $Y:\mathcal{Q}\to(0,\infty)$. The following two conditions on a weight $w\in L^1_{\loc}(\mathbb{R}^n,\d\mu)$  are equivalent:
\begin{enumerate}
\item There is some constant $B>0$ such that
\[
\frac{1}{Y(Q)}\int_Q|f(x)-f_{Q,\mu}|\,\d w(x)\leq B \|f\|_{\mathrm{BMO}(\d\mu)}
\]
for every function $f\in\mathrm{BMO}(\d\mu)$ and every cube $Q$ in $\mathbb{R}^n$.
\item The weight $w$ is an $A_{\infty,Y}(\d\mu)$ weight, i.e.
 \begin{equation}\label{eq:Fujii-WilsonY_mu}
 [w]_{A_{\infty,Y}(\d\mu)}:=\sup_{Q\in\mathcal{Q}}\frac{1}{Y(Q)}\int_Q M_\mu(w\chi_Q)(x)\,\d \mu(x)<\infty.
 \end{equation}
\end{enumerate}
Moreover, there exist positive constants $C_1$ and $C_2$ such that $C_1B \leq [w]_{A_{\infty,Y}(\d\mu)}\leq C_2 B$.
\end{theorem}

Note that $Y(Q):=w_r(Q)$, $r>1$ is a possible choice of $Y$ in the above theorem and thus the particular case of a constant functional in the main theorem \cite{Martinez2020} proves that $w$ is an $A_{\infty,w_r}(\d\mu)$ weight for any $r>1$. Evidently, the Fujii-Wilson $A_\infty(\d\mu)$ weights studied for instance in \cite{Hytoenen2012-2} are $A_{\infty,Y}(\d\mu)$ weights for the functional $Y$ defined by $Y(Q):=w(Q)$ for every cube $Q$ in $\mathbb{R}^n$. In particular, this answers the question on the need of the $A_\infty(\d\mu)$ condition for Theorem \ref{thm:Theorem1.5} at least in the case of a constant functional $a$. Indeed, on the one hand,  as weights satisfying the $A_\infty(\d\mu)$ condition \eqref{eq:(P3')} are precisely those satisfying the Fujii-Wilson $A_{\infty,w}(\d \mu)$ condition \eqref{eq:Fujii-WilsonY_mu}, Theorem \ref{thm:Theorem1.5} ensures that, for any weight satisfying the $A_\infty(\d\mu)$ condition  \eqref{eq:(P3')}, the self-improving inequality 
\[
\frac{1}{w(Q)}\int_Q|f(x)-f_{Q,\mu}|\,\d w(x)\leq B \|f\|_{\mathrm{BMO}(\d\mu)},\qquad Q\in\mathcal{Q}
\]
holds.
On the other hand, any weight for which the above self-improvement holds must be an $A_\infty(\d\mu)$ weight in virtue of Theorem \ref{pr:characterization_Ainfty} and \cite[Theorems 3.1 (b) and 4.2 (b)]{Duoandikoetxea2016}. Thus, according to Theorem \ref{thm:Theorem1.5}, it happens that the $SD_p^s(w)$ condition \eqref{eq:SDp} for the constant functionals is equivalent to the $A_\infty(\d\mu)$ condition on the weight $w$, i.e. $w\in A_\infty(\d\mu)$ if and only if there are $s>0$ and $C>0$ such that given a cube $Q$ in $\mathbb{R}^n$ and $\{Q_j\}_{j\in\mathbb{N}}\in\Delta(Q)$,
\begin{equation}\label{eq:SDpAinfty}
\left(\sum_{j\in\mathbb{N}} \frac{w(Q_j)}{w(Q)}\right)^{1/p}\leq C\left(\frac{\mu\left(\bigcup_{j\in\mathbb{N}} Q_j\right)}{\mu(Q)}\right)^{1/s}
\end{equation}
for some $p\geq 1$ (or equivalently, for every $p\geq1$). In fact, in this case it happens that $[w]_{A_\infty(\d\mu)}\asymp s/p$, where $s$ is the best possible exponent in the above condition.  In general, it is considered in \cite{Ombrosi2019} a general condition in the spirit of \eqref{eq:SDpAinfty} which generalizes the situation to more general functionals $Y$ (including the case $Y(Q):=w_r(Q)$, $r>1$) and which reads as follows: given $p\geq 1$, there is $s>0$ such that for any cube $Q$ in $\mathbb{R}^n$ and any $\{Q_j\}_{j\in\mathbb{N}}\in\Delta(Q)$, 
\begin{equation}\label{eq:generalized_Ainfty}
\left(\sum_{j\in\mathbb{N}}\frac{Y(Q_j)}{Y(Q)}\right)^{1/p}\leq C\left(\frac{\mu\left(\bigcup_{j\in\mathbb{N}} Q_j\right)}{\mu(Q)}\right)^{1/s}.
\end{equation}
This condition may be regarded as an $A_\infty(\d\mu)$ condition at scale $p$ for the functional $Y$ where, in analogy with the case $Y(Q):=w(Q)$, one could call $[Y]_{A_\infty(\d\mu,p)}$ (or $[Y]_{A_\infty(\d\mu)}$ in case $p=1$) to the best possible $s$  in the above condition. Observe that this generalizes the usual case, where for an $A_\infty(\d\mu)$ weight we have that $[w]_{A_\infty(\d\mu,p)}\asymp p[w]_{A_\infty(\d\mu)}$. Also, note that, by taking into account properties \eqref{w_r2} and \eqref{w_r1} of $w_r$,
\[
\begin{split}
\sum_{j\in\mathbb{N}}\frac{w_r(Q_j)}{w_r(Q)}\leq  \frac{w_r\left(\bigcup_{j\in\mathbb{N}}Q_j\right)}{w_r(Q)}\leq \left(\frac{\mu\left(\bigcup_{j\in\mathbb{N}}Q_j\right)}{\mu(Q)}\right)^{1/r'},
\end{split}
\]
so $w_r\in A_\infty(\d\mu)$ and $[w_r]_{A_\infty(\d\mu)}\leq r'$ for every $r>1$. This is the model example for the embedding result \cite[Theorem 1.6]{Ombrosi2019}, which is a particular case of \cite[Theorem 2]{Martinez2020}. 

As we advanced in the Introduction, it is our goal in this paper to get self-improving inequalities in the spirit of that in Theorem \ref{thm:PerezRela} (which, as already said, is the corollary of Theorem \ref{thm:Theorem1.5} we are interested in) replacing the $L^p$ norms by different norms or even quasi-norms.   Therefore, a brief reminder of the main concepts on the theory of quasi-normed spaces of functions is in order. 
 
\begin{definition}\label{def:quasinorm}
Let $X$ be a vector space. A function $\|\cdot\|:X\to[0,\infty)$ is called a quasi-norm if there is a constant $K\geq 1$ such that
\begin{enumerate}
\item $\|x\|=0$ if and only if $x=0$.
\item $\|\lambda x\|=|\lambda|\|x\|$ for $\alpha\in\mathbb{R}$ and $x\in X$.
\item $\|x_1+x_2\|\leq K\left(\|x_1\|+\|x_2\|\right)$ for all $x_1, x_2\in X$.  The constant $K$ will be called the \emph{geometric constant} of $\|\cdot\|$.
\end{enumerate}
A quasi-norm $\|\cdot\|$ over a vector space $X$ will be denoted by $\|\cdot\|_X$. In case $K=1$, the term ``quasi'' for the notation will be skipped. 
\end{definition}
 Consider now the measure space $(\mathbb{R}^n,\d\nu)$, where $\nu$ is a measure on the space. If $L^0(\mathbb{R}^n,\d\nu)$ is the vector lattice of all measurable functions modulo $\nu$-null functions, the positive cone of $L^0(\mathbb{R}^n,\d\nu)$ will be denoted by $L^0(\mathbb{R}^n,\d\nu)^+$. If $X(\d\nu)$ is an order ideal of $L^0(\mathbb{R}^n,\d\nu)$ 
(i.e. a vector subspace of $L^0(\mathbb{R}^n,\d\nu)$ such that $f\in X(\d\nu)$ for any $f\in L^0(\mathbb{R}^n,\d\nu)$ satisfying $|f|\leq |g|$ $\nu$-a.e. with $g\in X(\d\nu) $), a quasi-norm $\|\cdot\|_{X(\d\nu)}$ on $X(\d\nu)$ is said to be a lattice quasi-norm if $\|f\|_{X(\d\nu)}\leq \|g\|_{X(\d\nu)}$ 
whenever $f,g\in X(\d\nu)$ satisfy $|f|\leq |g|$. In this case, the pair $(X(\d\nu),\|\cdot\|_{X(\d\nu)})$ (or, sometimes, simply $X(\d\nu)$) is called a quasi-normed function space based on $(\mathbb{R}^n,\d\nu)$. For a given measurable subset $E$ of $\mathbb{R}^n$, the notation $\|f\|_{X(E,\d\mu)}:=\|f \chi_E\|_{X(\d\mu)}$ will be used.   Recall also the discussion on the concept of  local average below Theorem \ref{thm:PerezRela} in the Introduction.

The  normed function spaces introduced here coincide with those called normed K\"othe function spaces \cite[Ch. 15]{Zaanen1967}, which are defined as those for which a function norm $\rho:\mathcal{M}^+(\nu)\to[0,\infty]$ is finite, where $\mathcal{M}^+(\nu)$ is the class of nonnegative measurable functions up to $\nu$-a.e. null functions.  See \cite[Remark 2.3 (ii)]{Okada2008} for more {details} about this. 

If we want to study self-improving results in the spirit of Theorem \ref{thm:PerezRela} (or, more in general, in the spirit of \cite[Theorem 1.6]{Ombrosi2019}) for norms different from the $L^p$ ones, a good strategy would be to try to write the conditions on these theorems in terms of the $L^p$ norm. If the obtained result makes sense for a different norm, it may be the correct condition for such a generalization. It turns out  that this strategy works. Indeed, consider the  general condition \eqref{eq:generalized_Ainfty} and pick any family $\{h_j\}_{j\in\mathbb{N}}$ of functions satisfying that $ \left\|\ h_j  \chi_{Q_j}\right\|_{L^p\left(Q_j,\frac{\d w}{Y(Q_j)}\right)}=1$, where $\{Q_j\}_{j\in\mathbb N}$ is a family of pairwise disjoint subcubes of a cube $Q$.
 We can make the following computations:  
\[
\begin{split}
\left(\sum_{j\in\mathbb{N}}\frac{Y(Q_j)}{Y(Q)}\right)^{1/p}&= \left(\sum_{j\in\mathbb{N}} \frac{1}{Y(Q)}\int_{Q_j} h_j(x)^p\,\d w(x) \right)^{1/p} \\
&= \left( \frac{1}{Y(Q)}\int_{Q }\sum_{j\in\mathbb{N}} h_j(x)^p\chi_{Q_j}(x)\,\d w(x) \right)^{1/p} \\
&= \left( \frac{1}{Y(Q)}\int_{Q} \left[\sum_{j\in\mathbb{N}} h_j(x) \chi_{Q_j}(x)\right]^p\,\d w(x) \right)^{1/p} \\
&= \left\|\sum_{j\in\mathbb{N}} h_j  \chi_{Q_j}\right\|_{L^p\left(Q,\frac{\d w}{Y(Q)}\right)},
\end{split}
\]
This way, we have written the left-hand side of \eqref{eq:generalized_Ainfty} in terms of the $L^p(\d w)$ norm. This, and the fact that in the self-improving results there is no special reason why this left-hand side must be controlled by a power function of $\mu\left(\bigcup_{j\in\mathbb{N}}Q_j\right)/\mu(Q)$, leads us to make the following definition, in a clear paralellism with the comments below \eqref{eq:generalized_Ainfty}.

\begin{definition}\label{def:generalized_Ainfty2}
Let $\mu,\nu$ be measures in $\mathbb{R}^n$ (usually a weighted measure) and consider a quasi-norm $\|\cdot\|_{X(\d\nu)}$. A functional $Y:\mathcal{Q}\to(0,\infty)$ will be said to be an $A_\infty(\d\mu, X(\d\nu))$ functional if there exist some constant ${C_Y}>0$ and some increasing bijection $\Psi:[0,1]\to[0,1]$ such that
\begin{equation}\label{eq:general_smallness}
 \left\|\sum_{j\in\mathbb{N}}  h_j \chi_{Q_j }\right\|_{X\left(Q,\frac{\d \nu}{Y(Q)}\right)}\leq  {C_Y}\Psi^{-1}\left[\frac{\mu\left(\bigcup_{j\in\mathbb{N}} Q_j\right)}{\mu(Q)}\right]
\end{equation}
for every $\{Q_j\}_{j\in\mathbb{N}}\in\Delta(Q)$, $Q\in\mathcal{Q}$ and every    family of functions $\{h_j\}_{j\in\mathbb{N}}$ satisfying $\|h_j\|_{X\left(Q_j,\frac{\d \nu}{Y(Q_j)}\right)}=1$ for every $j\in \mathbb{N}$.
\end{definition}
This condition generalizes the above $A_\infty(\d\mu)$ condition \eqref{eq:generalized_Ainfty} for functionals to any abstract quasi-norm $\|\cdot\|_{X(\nu)}$ scale.

With this condition  at hand, it is possible to prove a new self-improving result which generalizes \cite[Theorem 1.5]{Perez2019}, \cite[Theorem 1.6]{Ombrosi2019} and \cite[Theorem 2]{Martinez2020} in the case a constant functional $a$ is considered. First some technical lemmas have to be proved. We will start by imposing some conditions on the function norms we will consider. 
\begin{definition}\label{def:compatible} Let   $Y:\mathcal{Q}\to(0,\infty)$ be a functional and consider a measure $\nu$ in $\mathbb{R}^n$. A lattice quasi-norm $\|\cdot\|_{X(\d\nu)}$    will be said to be good with respect to $Y$   if it satisfies:
\begin{enumerate}
\item (Fatou's property) \label{property:general_Fatou}  If $\{f_k\}_{k\in\mathbb{N}}$ are positive functions in $X(\d\nu)$ with $f_k\uparrow f$ $\nu$-a.e. for some function $f\in X(\d\nu)$, then $\|f_k\|_{X(\d\nu)}\uparrow \|f\|_{X(\d\nu)}$ and $\|f_k\|_{X\left(Q,\frac{\d\nu}{Y(Q)}\right)}\uparrow \|f\|_{X\left(Q,\frac{\d\nu}{Y(Q)}\right)}$ for any cube $Q$.
\item \label{property:consistency} $\chi_E\in X(\d\nu)$ for any $\nu$-finite measure set $E$.
\item (Average property) \label{property:general_average}  $\|\chi_Q\|_{X\left(Q,\frac{\d\nu}{Y(Q)}\right)}\leq 1$ for every cube $Q$ in $\mathbb{R}^n$. This will be called the average property of $\|\cdot\|_{X(\d\nu)}$ with respect to $Y$.
\end{enumerate}
If $\|\cdot\|_{X(\d\nu)}$   is good with respect to $Y$ it will be said that they are compatible. If $\|\cdot\|_{X(\d\nu)}$   is good with respect to $Y(Q):=\nu(Q)$ we will simply say that it is a good quasi-norm.
\end{definition}
\begin{example}
The $L^p(\mathbb{R}^n,\d\nu)$ space , $p\geq1$, is  an example of space with a norm compatible with any functional $Y$ satisfying that $\nu(Q)\leq Y(Q)$ for every cube $Q$ in $\mathbb{R}^n$.   Other typical examples are the weak Lebesgue spaces, defined for $0<p<\infty$ as
\[
L^{p,\infty}(\mathbb{R}^n,\d\nu):=\left\{ f\in L^0(\mathbb{R}^n,\d\nu):\|f\|_{L^{p,\infty}(\mathbb{R}^n,\d\nu)}<\infty  \right\},
\]
where  $\nu$ can be the usual underlying doubling measure $\mu$ or any other measure depending or not on $\mu$.  Here we use the standard notation $\|f\|_{L^{p,\infty}(\mathbb{R}^n,\d\nu)}$ for the \emph{weak norm} defined as
\begin{equation*}
\|f\|_{L^{p,\infty}(\mathbb{R}^n,\d\nu)}:= \sup_{t>0}t \nu\left(\{x\in\mathbb{R}^n:|f(x)|>t\}\right)^{\frac{1}{p}}.
\end{equation*}
 
\end{example}

\section{A new quantitative self-improving theorem for \texorpdfstring{$\mathrm{BMO}$}{BMO} functions}\label{sec:self-improve}

In this section we prove the new general self-improving result in Theorem \ref{thm:self-improving_general}. We start by proving some preliminary lemmas which will allow us to reduce  the proof to the case of bounded functions.

\subsection{Lemmata}

We first include the following trivial lemma regarding the oscillations of a function.
\begin{lemma}\label{lem:osc}
Let $f\in L^1_{\mathrm{loc}}(\mathbb{R}^n,d\mu)$ and let $p\geq 1$. If $E$ is a positive finite measure set of $\mathbb{R}^n$, then
\[
\begin{split}
\inf_{c\in \mathbb{R}}\left(\frac{1}{\mu(E)}\int_E |f(x)-c|^pd\mu(x)\right)^{1/p}&\leq\left(\frac{1}{\mu(E)}\int_E |f(x)-f_E|^pd\mu(x)\right)^{1/p}\\
&\leq 2\inf_{c\in \mathbb{R}}\left(\frac{1}{\mu(E)}\int_E |f(x)-c|^pd\mu(x)\right)^{1/p}.
\end{split}
\]
\end{lemma}

Recall that, for given $L<U$, the notation $\tau_{LU}$ is used for the  function   $\tau_{LU}:\mathbb{R}\to \mathbb{R}$  given by
\[
\tau_{LU}(a):=\begin{cases}
L&\text{if }a<L,\\
a&\text{if  }L\leq a\leq U\\
U&\text{if }a>U.
\end{cases}
\]

These functions allow  to define the truncations  $\tau_{LU}(g)$ of a given function $g$ by
\[
\tau_{LU}g(x):=\tau_{LU}(g(x)),\qquad  L<U,\ x\in\mathbb{R}^n.
\]

\begin{lemma}\label{lem:osc2_general}
Let $\nu$ be any  Borel measure in $\mathbb{R}^n$ and consider $f\in L^1_{\loc}(\mathbb{R}^n,\d\nu)$. Then, for every cube $Q$ in $\mathbb{R}^n$,   
\[
\frac{1}{\nu(Q)}\int_Q|f-f_{Q,\nu}|\,\d\nu  \leq\sup_{L<U}\frac{1}{\nu(Q)}\int_Q|\tau_{LU}f-(\tau_{LU}f)_{Q,\nu}|\,\d\nu\leq \frac{2}{\nu(Q)}\int_Q|f-f_{Q,\nu}|\,\d\nu.
\]
 \end{lemma}
\begin{proof}

Let $Q$ be a cube in $\mathbb{R}^n$. Observe first that, given $L<U$ one has that $|\tau_{LU}(a)-\tau_{LU}(b)|\leq |a-b|$ for every $a,b\in \mathbb{R}$. This allows to write 
\[
\begin{split}
\frac{1}{\nu(Q)}\int_Q|\tau_{LU}f-(\tau_{LU}f)_{Q,\nu}|\,\d\nu&\leq 2 \inf_{c\in\mathbb{R}}\frac{1}{\nu(Q)}\int_Q|\tau_{LU}f-c|\,\d\nu\\
 &  \leq \frac{2}{\nu(Q)}\int_Q|\tau_{LU}f-\tau_{LU}(f_{Q,\nu})|\,\d\nu \\
 & \leq   \frac{2}{\nu(Q)}\int_Q|f-f_{Q,\nu}|\,\d\nu,
\end{split}
\]
for every $L<U$. Here Lemma \ref{lem:osc} has been used.

On the other hand, by Fatou's lemma,  
\[
\begin{split}
 \frac{1}{\nu(Q)}\int_Q|f-f_{Q,\nu}|\,\d\nu&\leq \liminf_{\substack{ L\to-\infty,\\   U\to\infty}} \frac{1}{\nu(Q)}\int_Q|\tau_{LU}f-(\tau_{LU}f)_{Q,\nu}|\,\d\nu\\
&\leq \sup_{L<U}\frac{1}{\nu(Q)}\int_Q|\tau_{LU}f-(\tau_{LU}f)_{Q,\nu}|\,\d\nu, 
\end{split}
\]
 and the result will follow.    Here  the local integrability of $f$ was used to ensure $f_{Q,\nu}=\lim_{\substack{ L\to-\infty,\\   U\to\infty}}(\tau_{LU}f)_{Q,\nu}$ by dominated convergence.
\end{proof}
 
\begin{lemma}\label{lem:osc2_general2}
Let $\mu,\nu$ be  Borel measures in $\mathbb{R}^n$ and let $f\in L^1_{\loc}(\mathbb{R}^n,\d \mu)$. Consider a lattice quasi-norm  $\|\cdot\|_{X(\d\nu)}$ compatible with  a functional $Y:\mathcal{Q}\to(0,\infty)$.   Then, for every cube $Q$ in $\mathbb{R}^n$,   
\[
\|f-f_{Q,\mu}\|_{X\left(Q,\frac{\d\nu}{Y(Q)}\right)}\leq\sup_{L<U}\|\tau_{LU}f-(\tau_{LU}f)_{Q,\mu} \|_{X\left(Q,\frac{\d\nu}{Y(Q)}\right)}.
\]

\end{lemma}
\begin{proof}
Let $Q$ be a cube in $\mathbb{R}^n$.   By Fatou's property \ref{property:general_Fatou})  in Definition \ref{def:compatible},
\[
\begin{split}
\|f -f_{Q,\mu }\|_{X\left(Q,\frac{\d\nu}{Y(Q)}\right)} &\leq \liminf_{\substack{ L\to-\infty,\\   U\to\infty}} \|\tau_{LU}f -(\tau_{LU}f)_{Q,\mu}\|_{X\left(Q,\frac{\d\nu}{Y(Q)}\right)} \\
&\leq  \sup_{L<U}\|\tau_{LU}f -(\tau_{LU}f)_{Q,\mu}\|_{X\left(Q,\frac{\d\nu}{Y(Q)}\right)}, 
\end{split}
\]
 and the result will follow.    Here  the local integrability of $f$ was used to ensure $f_{Q,\mu}=\lim_{\substack{ L\to-\infty,\\   U\to\infty}}(\tau_{LU}f)_{Q,\mu}$ by dominated convergence.
\end{proof}

\subsection{Proof of the self-improving theorem}
We are already in position to present  the proof of our   Theorem \ref{thm:self-improving_general}. 

 \begin{proof} Lemmas \ref{lem:osc2_general} and \ref{lem:osc2_general2} allow  to work under the assumption that $f$ is a bounded function. Since   $f$ is in $\mathrm{BMO}(\d\mu)$, for every cube $P$ in $\mathbb{R}^n$, the following inequality holds \begin{equation}\label{eq:startM2}
 \frac{1}{\mu(P)}\int_P \frac{|f(x)-f_{P,\mu}|}{\|f\|_{\mathrm{BMO}(\d\mu)}}\,\d\mu(x)\leq 1.
 \end{equation}
Let $L>1$ and let  $Q$ be any cube in $\mathbb{R}^n$. Inequality \eqref{eq:startM2} allows to apply the local Calder\'on-Zygmund decomposition to $\frac{f(x)-f_{Q,\mu}}{\|f\|_{\mathrm{BMO}(\d\mu)}}$ on $Q$  at level $L$. This gives a family of disjoint subcubes $\{Q_j\}_{j\in\mathbb{N}}\subset \mathcal{D}(Q)$ with the properties
\begin{equation}\label{CZ1M2}
L < \frac{1}{\mu(Q_j)}\int_{Q_j}\frac{|f(x)-f_{Q,\mu}|}{\|f\|_{\mathrm{BMO}(\d\mu)}}\,\d\mu(x) \leq c_\mu 2^{n_\mu} L.
\end{equation}
 For a simpler presentation, let us introduce the notation
\begin{equation*}
g(x):=\frac{f(x)-f_{Q,\mu}}{\|f\|_{\mathrm{BMO}(\d\mu)}}
\qquad\text{and} \qquad
g_j(x):=\frac{f(x)-f_{Q_j,\mu}}{\|f\|_{\mathrm{BMO}(\d\mu)}}.
\end{equation*}

The function $g\chi_Q$ can be decomposed as 
 \[
\begin{split}
g(x)\chi_Q(x)&=\sum_{j\in\mathbb{N}}g(x)\chi_{Q_j}(x)+g(x)\chi_{Q\backslash \bigcup_{j\in\mathbb{N}} Q_j}(x)\\
&=\sum_{j\in\mathbb{N}}\left[g_j(x)+\frac{f_Q-f_{Q_j,\mu}}{\|f\|_{\mathrm{BMO}(\d\mu)}}\right]\chi_{Q_j}(x)+g(x)\chi_{Q\backslash \bigcup_{j\in\mathbb{N}} Q_j}(x).
\end{split}
\]
On one hand, by Lebesgue differentiation theorem 
\[
\left|g(x)\chi_{Q\backslash \bigcup_{j\in\mathbb{N}} Q_j}(x)\right|\leq L,
\]
for $\mu$-almost every $x\in Q$ and, on the other hand, the second term in the sum 
$$
\sum_{j\in\mathbb{N}}\left[g_j(x)+\frac{f_Q-f_{Q_j,\mu}}{\|f\|_{\mathrm{BMO}(\d\mu)}}\right]\chi_{Q_j}(x)
$$ 
can be bounded as follows
\[
\left|\frac{f_Q-f_{Q_j,\mu}}{\|f\|_{\mathrm{BMO}(\d\mu)}}\right|\leq \frac{1}{\mu(Q_j)}\int_{Q_j}\frac{|f(x)-f_{Q,\mu}|}{\|f\|_{\mathrm{BMO}(\d\mu)}}\d\mu(x) \leq c_\mu  2^{n_\mu} L,
\]
for every $j\in\mathbb{N}$.

Therefore, the absolute value of $g$ can be bounded by 
\[
\begin{split}
\left|g(x)\right|\chi_Q(x)&\leq \sum_{j\in\mathbb{N}}\left| g_j(x)\right| \chi_{Q_j}(x)+ (c_\mu 2^{n_\mu}+1)L\chi_Q(x).
\end{split}
\]
Hence,  by using the  quasi-triangle inequality, the Average property \ref{property:general_average} from Definition \ref{def:compatible} and the disjointness of the cubes $Q_j$,
\begin{equation}\label{eq:main-argument-CZ}
 \left\| g\right\|_{X\left(Q,\frac{\d w}{Y(Q)}\right)}\leq  K   \left\|\sum_{j\in\mathbb{N}} g_j\chi_{Q_j}\right\|_{X\left(Q,\frac{\d w}{Y(Q)}\right)}+C_{K,\mu}L
\end{equation}
 where $C_{K,\mu}:=K(c_\mu2^{n_\mu}+1)$.

The key property of the cubes $\{Q_j\}_{j\in\mathbb{N}}$ in the Calder\'on-Zygmund decomposition at level $L$ of $ g(x)=\frac{f(x)-f_{Q,\mu}}{\|f\|_{\mathrm{BMO}(\d\mu)}} \chi_Q(x)$ is the fact that, by \eqref{CZ1M2},
\begin{equation}\label{eq:CZ_smallness}
\sum_{j\in\mathbb{N}}\mu(Q_j) \leq \sum_{j\in\mathbb{N}}\frac{1}{L}\int_{Q_j}|g(x)|\,\d\mu(x)=\frac{1}{L}\int_{Q}|g(x)|\,\d\mu(x)\leq \frac{\mu(Q)}{L},
\end{equation}
where \eqref{eq:startM2} has been used.
 
 A brief remark is in order here to explain the main idea. Note that in \eqref{eq:main-argument-CZ} we have essentially the same object on both sides of the inequality but at different levels. That is, we are trying to control the local average $\left\| g\right\|_{X\left(Q,\frac{\d w}{Y(Q)}\right)}$ in terms of a local average of  $\sum_j g_j\chi_{Q_j}$.
In the classical case of $L^p$ weighted norms, the localization of the functions $g_j\chi_{Q_j}$ allows to move the norm into the sum. Here, however, we will appeal to the generalized $A_\infty$ condition for the functional $Y$. To that end, let us define 
\begin{equation}\label{eq:X_varepsilonM2}
\mathbb{X} := \sup_{P\in\mathcal{Q}}\left\| \frac{f -f_{P,\mu}}{\|f\|_{\mathrm{BMO}(\d\mu)})}\right\|_{X\left(P,\frac{\d w}{Y(P)}\right)}. 
\end{equation}

This supremum is finite since, by the Average property \ref{property:general_average}  from Definition \ref{def:compatible}  and  the boundedness of $f$,  for any cube  $P\in\mathcal{Q}$,
\[
\left\| \frac{f -f_{P,\mu}}{\|f\|_{\mathrm{BMO}(\d\mu)}}\right\|_{X\left(P,\frac{\d w}{Y(P)}\right)}\leq 2\frac{\|f\|_{L^\infty(\mathbb{R}^n,\d\mu)}}{\|f\|_{\mathrm{BMO}(\d\mu)}}<\infty.
\]
This allows to make computations with $\mathbb{X}$, which allows to introduce the local averages of the functions $g_j\chi_{Q_j}$  as follows:

\begin{eqnarray*}
 \left\| g\right\|_{X\left(Q,\frac{\d w}{Y(Q)}\right)}&\leq   & K  \left\|\sum_{j\in\mathbb{N}} g_j\chi_{Q_j}\right\|_{X\left(Q,\frac{\d w}{Y(Q)}\right)}+ C_{K,\mu}\cdot L\\ 
 &= & K\left\|\sum_{j\in\mathbb{N}}
 \frac{ \left\| g_j\right\|_{X\left(Q_j,\frac{\d w}{Y(Q_j)}\right)}}{ \left\| g_j\right\|_{X\left(Q_j,\frac{\d w}{Y(Q_j)}\right)}}
g_j\chi_{Q_j}\right\|_{X\left(Q,\frac{\d w}{Y(Q)}\right)}+ C_{K,\mu}\cdot L\end{eqnarray*}

Using the $\mathbb{X}$ defined above, we get 
\begin{equation}\label{boundM2}
\left\| g\right\|_{X\left(Q,\frac{\d w}{Y(Q)}\right)}\le \mathbb{X}\cdot K
 \left\|\sum_{j\in\mathbb{N}}
 \frac{ g_j}{ \left\| g_j\right\|_{X\left(Q_j,\frac{\d w}{Y(Q_j)}\right)}}
\chi_{Q_j}\right\|_{X\left(Q,\frac{\d w}{Y(Q)}\right)}+ C_{K,\mu}\cdot L
\end{equation}

 Here is where the $A_\infty(\d\mu,X(w))$ condition for $Y$ pops in. In particular, we get a  bound  that does not depend on the cube $Q$ (recall that $\{Q_j\}_{j\in\mathbb{N}}$ and $Q$ satisfy the smallness relation \eqref{eq:CZ_smallness}), and then  one can take supremum at the left-hand side to get
	
\[
\mathbb{X}\leq \mathbb{X}\cdot {C_Y}\cdot K\cdot{\Psi}^{-1}\left(\frac{1}{L}\right)+ C_{K,\mu}\cdot L,
\]
where ${C_Y}$ is the constant in the definition of $A_\infty(\d\mu,X(w))$.
One can now choose $L> \max\left\{1,\left[ {\Psi} \left((C_Y\cdot K)^{-1} \right)\right]^{-1}\right\}$. Thanks to this, it is possible to isolate $\mathbb{X}$ at the left-hand side as follows
\[
\mathbb{X}  \left[1-{C_Y}\cdot K\cdot {\Psi}^{-1}\left(\frac{1}{L}\right)\right]\leq  C_{K,\mu}L.
\]
Equivalently,
\[
\mathbb{X} \leq  C_{K,\mu}\frac{L }{ 1-{C_Y} \cdot K\cdot{\Psi}^{-1}\left(\frac{1}{L}\right) }
\]
for every  $L>\max\left\{1,\left[{\Psi} \left(({C_Y}\cdot K)^{-1} \right)\right]^{-1}\right\}$. It just remains to optimize the right-hand side on $L>\max\left\{1,\left[ {\Psi}\left(({C_Y}\cdot K)^{-1} \right)\right]^{-1}\right\}$ to get the desired result.
 \end{proof}
\begin{remark}
{Assume that $Y$ is a functional that satisfies the hypothesis of the theorem we have just proved, and additionally satisfies the  following property: if $Q$ is any cube in $\mathbb{R}^n$, then 
\[
\frac{1}{Y(Q)}\int_Q f(x)\,\d\nu(x)\leq \|f\|_{X\left(Q,\frac{\d\nu}{Y(Q)}\right)}.
\]
That additional property combined with the Theorem just settled and Theorem \ref{pr:characterization_Ainfty} yields that whenever  $Y$ is an $ A_\infty(\d\mu,X(w))$ functional in the sense of Definition \ref{def:generalized_Ainfty2}, it happens that the weight $w$ satisfies the Fujii-Wilson type $A_{\infty,Y}(\d\mu)$ condition in item (2) of Theorem \ref{pr:characterization_Ainfty}.}
\end{remark}
 
 \section{Applications of the  self-improving theorem}\label{sec:applications}
 
As a first easy consequence of our general self-improving result, we include the following corollary regarding the classical $A_\infty$ condition. We show that it suffices to check the usual condition replacing the usual power functions by any increasing bijection. We remit the reader to \cite{Duoandikoetxea2016} for the classical definition and a number of equivalent conditions.
\begin{corollary}\label{cor:Ainfty-general}
Let $w$ be any weight satisfying, for some increasing bijection $\Phi:[0,1]\to[0,1]$, the condition 
 \[
 \frac{w\left(\bigcup_{j\in\mathbb{N}}Q_j\right)}{w(Q)}\leq C\Phi^{-1}\left[ \frac{\mu\left(\bigcup_{j\in\mathbb{N}}Q_j\right)}{\mu(Q)}\right]
 \]
 for every cube $Q$ and every $\{Q_j\}_{j\in\mathbb{N}}\in\Delta(Q)$. Then $w\in A_\infty(\d\mu)$.
\end{corollary}
\begin{proof} 
Consider $X(w):=L^p(\d \mu)$ and $Y=w$ for a weight $w\in L^1_{\loc}(\mathbb{R}^n,\d\mu)$. By Theorem \ref{thm:self-improving_general}, we get that the weight $w$ satisfies also  that 
 \[
 \sup_{\|f\|_{\mathrm{BMO}(\d\mu)}=1} \|f-f_{Q,\mu}\|_{L^p\left(Q,\frac{\d w}{w(Q)}\right)}<\infty,
 \]
 and so, by Theorem \ref{pr:characterization_Ainfty}, it happens that $w\in A_\infty(\d\mu)$. 
\end{proof}

 In the sequel, we present two particular examples of application of Theorem \ref{thm:self-improving_general}.

\subsection{\texorpdfstring{$\mathrm{BMO}$}{BMOtype}-type improvement at the Orlicz spaces scale} 
 
The first example has to do with Orlicz norms for submultiplicative Young functions. The aim is to write a quantitative self-improving result for the control on the mean oscillations of  $\mathrm{BMO}(\d\mu)$ functions to a control on Orlicz mean oscillations. 

A special type of convex function is used to define Orlicz norms.  
  
\begin{definition}
A convex function $\phi:[0, \infty)\rightarrow[0, \infty)$ is said to be a Young function if $\phi(0)=0$, $\phi(1)=1$ and $\lim_{t\to\infty}\phi(t)=\infty$.  We will say that $\phi$ is a quasi-submultiplicative Young function with associated constant $c>0$ if, additionally, $\phi(t_1\cdot t_2)\leq c\phi(t_1)\cdot \phi(t_2)$ for  every $t_1,t_2\geq0$.  If $c=1$ we will simply say that $\phi$ is submultiplicative. If there is $k>2$ such that $\phi(2t)\leq k\phi(t)$ for every $t\geq t_0$ for some $t_0\geq 0$, we will say that $\phi$ satisfies the $\Delta_2$ (or doubling) condition.
\end{definition}

\begin{example}
As examples of doubling Young functions one can find the power functions $\phi_p(t):=t^p$.  These are clearly doubling functions since they are submultiplicative. In general, every submultiplicative Young function $\phi$ is a doubling Young function but not only submultiplicative functions satisfy this condition, as this is also fulfilled by quasi-submultiplicative  Young functions such as $\phi_{p,\alpha}(t):=\log(e+1)^{-\alpha}t^p\log(e+t)^\alpha$, $p\geq1$, $\alpha>0$.
\end{example}

  Given any Young function $\phi$, any Borel measure $\nu$ in $\mathbb{R}^n$  and any cube $Q$ in $\mathbb{R}^n$ one can define the $\phi(L)(\nu)$-mean average of a function $f$ over $Q$ with the Luxemburg norm 
\[
\|f\|_{\phi(L)\left(Q,\frac{\d\nu}{\nu(Q)}\right)}:=\inf\left\{\lambda>0:\frac{1}{\nu(Q)}\int_Q \phi\left(\frac{|f(x)|}{\lambda}\right)\,\d\nu(x)\leq 1\right\},
\] 
 which is the localized version of the Luxemburg norm defining the Orlicz space $\phi(L)(\mathbb{R}^n,\d\nu)$ given by the finiteness of the norm 
 \[
\|f\|_{\phi(L)\left(\mathbb{R}^n,\d\nu\right)}:=\inf\left\{\lambda>0: \int_{\mathbb{R}^n} \phi\left(\frac{|f(x)|}{\lambda}\right)\,\d\nu(x)\leq 1\right\}.
\] 
 
Note that if $\phi_1,\phi_2$ are   Young functions satisfying $\phi_1(t)\leq \phi_2(kt)$ for $t>t_0$, for some $k>0$ and $t_0\geq0$, then $\phi_2(L)\left(Q,\frac{\d\nu}{\nu(Q)}\right)\subset \phi_1(L)\left(Q,\frac{\d\nu}{\nu(Q)}\right)$ (see \cite[Theorem 13.1]{Krasnoselskii1961}). Therefore one can find infinitely many Orlicz spaces different from a Lebesgue space between any two Lebesgue spaces $L^p\left(Q,\frac{\d\nu}{\nu(Q)}\right)$ and $L^q\left(Q,\frac{\d\nu}{\nu(Q)}\right)$, $p<q$.

Orlicz spaces are examples of quasi-normed function spaces with a good quasi-norm as introduced in the beginning of this section and moreover they are Banach function spaces, i.e. the quasi-norm $\|\cdot\|_{\phi(L)\left(\mathbb{R}^n, \d\nu \right)}$ is in fact a norm and the resulting space is complete. 

The following lemma shows that, in the Orlicz setting and for the particular case $Y(Q):=\mu(Q)$, it suffices to check the $A_\infty(\d\mu, \phi(L)(\d \mu))$ condition in Definition \ref{def:generalized_Ainfty2}  just for characteristic functions instead of considering the arbitrary functions $h_j$.
\begin{lemma}\label{lem:reduction} 
Let $\phi$ be a  quasi-submultiplicative Young function with associated quasi-submultiplicative constant $c>0$. The functional $Y:\mathcal{Q}\to(0,\infty)$ given by $Y(Q):=\mu(Q)$ is an  $A_\infty(\d\mu, \phi(L)(\d \mu))$ functional with associated increasing bijection $\Phi:[0,1]\to[0,1]$   if and only if there is $C>0$ such that
\begin{equation}\label{eq:general_smallness-char}
 \left\|\sum_{j\in\mathbb{N}}   \chi_{Q_j } \right\|_{\phi(L)\left(Q,\frac{\d \mu}{\mu(Q)}\right)}\leq  C\Phi^{-1}\left[\frac{\mu\left(\bigcup_{j\in\mathbb{N}} Q_j\right)}{\mu(Q)}\right]
\end{equation}
for every $\{Q_j\}_{j\in\mathbb{N}}\in\Delta(Q)$, $Q\in\mathcal{Q}$. 
\end{lemma}
\begin{proof}  
 Indeed, consider a cube $Q$, a sequence $\{Q_j\}_{j\in\mathbb{N}}$ of disjoint subcubes of $Q$ and $\{h_j\}_{j\in\mathbb{N}}$ a sequence of functions satisfying $\|h_j\|_{\psi(L)\left(Q_j,\frac{\d\mu}{\mu(Q_j)}\right)}=1$ for every $j\in\mathbb{N}$. Then,
\[
\begin{split}
\frac{1}{\mu(Q)} \int_Q \phi\left(\frac{\sum_{j\in\mathbb{N}} h_j\chi_{Q_j} }{\lambda}\right)\,\d\mu(x)& =  \sum_{j\in\mathbb{N}} \frac{\mu(Q_j)}{\mu(Q)}\frac{1}{\mu(Q_j)}\dashint_{Q_j} \phi\left(\frac{h_j (x) }{\lambda}\right)\,\d\mu(x)\\
&\leq c\sum_{j\in\mathbb{N}} \frac{\mu(Q_j)}{\mu(Q)}\phi\left(\frac{ 1}{\lambda}\right)\frac{1}{\mu(Q_j)}\int_{Q_j}\phi(h_j(x)) \,\d\mu(x)\\
 &=  c\sum_{j\in\mathbb{N}}\frac{1}{\mu(Q)} \int_{Q_j}\phi\left(\frac{ 1 }{\lambda}\right)\,\d\mu(x)\\
 &\leq  c \frac{1}{\mu(Q)}\int_{Q}\phi\left(\frac{ \sum_{j\in\mathbb{N}} \chi_{Q_j}(x)  }{\lambda}\right)\,\d\mu(x).
\end{split}
\]
Hence, 
\[
\begin{split}
\left\|\sum_{j\in\mathbb{N}} h_j\chi_{Q_j}  \right\|_{\phi(L)\left(Q,\frac{\d\mu}{\mu(Q)}\right)} &\leq c \left\|\sum_{j\in\mathbb{N}}  \chi_{Q_j} \right\|_{\psi(L)\left(Q,\frac{\d\mu}{\mu(Q)}\right)}.
\end{split}
\]
The result follows from the above computation and the fact that characteristic functions have  average $1$.
\end{proof}

As a consequence, we get that the underlying measure $\mu$ defines an $A_\infty(\d\mu,\phi(L)(\d\mu))$ functional for any Orlicz norm induced by a quasi-submultiplicative Young function $\phi$.
\begin{lemma}\label{lem:OrliczGoodFunct}
Let $\phi$ be a  quasi-submultiplicative Young function with associated quasi-submultiplicative constant $c>0$. The functional $Y:\mathcal{Q}\to(0,\infty)$ given by 
$Y(Q):=\mu(Q)$ is an  $A_\infty(\d\mu, \phi(L)(\d \mu))$ functional with associated increasing bijection given by $\Psi(t):= 1/\phi^{-1}(1/t)$.
\end{lemma}
\begin{proof}
By the above lemma, we just have to check the condition for characteristic functions.  Let us then take a cube $Q$ and any family $\{Q_j\}_{j\in\mathbb{N}}\in\Delta(Q)$. Then, if one considers $\lambda_0:=1/\phi^{-1}\left[\mu(Q)/\sum_{j\in\mathbb{N}}\mu(Q_j)\right]$,
\[
\begin{split}
\frac{1}{\mu(Q)}\int_Q\phi\left(\frac{\sum_{j\in\mathbb{N}}\chi_{Q_j}(x)}{\lambda_0} \right)\,\d\mu(x) &= \sum_{j\in\mathbb{N}} \frac{1}{\mu(Q)}\int_{Q_j}\phi\left(\frac{\chi_{Q_j}(x)}{\lambda_0}\right)\,\d\mu(x)\\
&= \sum_{j\in\mathbb{N}}\frac{\mu(Q_j)}{\mu(Q)}\phi\left(\frac{1}{\lambda_0}\right)\\
&=\sum_{j\in\mathbb{N}}\frac{\mu(Q_j)}{\mu(Q)}\frac{\mu(Q)}{\sum_{j\in\mathbb{N}}\mu(Q_j)}=1.
\end{split}
\]
This implies that, for any cube $Q$ and any family $\{Q_j\}_{j\in\mathbb{N}}\in \Delta(Q)$,
\[
\begin{split}
\left\| \sum_{j\in\mathbb{N}}\chi_{Q_j} \right\|_{\phi(L)\left(Q,\frac{\d\mu}{\mu(Q)}\right)}&=\inf\left\{ \lambda>0:\frac{1}{\mu(Q)}\int_Q\phi\left(\frac{\sum_{j\in\mathbb{N}}\chi_{Q_j}(x)}{\lambda}\right)\,\d\mu(x) \right\}\\
&\leq \frac{1}{\phi^{-1}\left[\frac{\mu(Q)}{\sum_{j\in\mathbb{N}}\mu(Q_j)}\right]}.
\end{split}
\]
The smallness condition is then satisfied for the increasing bijection of $[0,1]$ given by $\Psi^{-1}(t):= 1/\phi^{-1}(1/t)$.

\end{proof}

From the above lemma it can be deduced,  through a simple application of Theorem \ref{thm:self-improving_general},  the following general result for Orlicz spaces.
 \begin{corollary}\label{cor:OrliczBMO}
 Let $\mu$ be a doubling measure in $\mathbb{R}^n$. Let $\phi$ be a quasi-submultiplicative Young function with associated quasi-submultiplicative constant $c>0$.  Let us further assume that $\phi$ is differentiable for $t>1$ and let $[\phi]_1$ and $[\phi]_2$ be the best constants satisfying $[\phi]_1\phi(t)\leq t\phi'(t)\leq [\phi]_2\phi(t)$, $t>1$.  If $f\in \mathrm{BMO}(\d\mu)$  then  
 \begin{equation}
 \| f-f_{Q,\mu}\|_{\phi(L)\left(Q,\frac{\d\mu}{\mu(Q)}\right)}\leq   c_\mu 2^{n_\mu}        \phi\left[c\left(1+\frac{1}{[\phi]_1}\right)\right]\left([\phi]_2+1\right) \|f\|_{\mathrm{BMO}(\d\mu)}
 \end{equation}
 for every cube $Q$ in $\mathbb{R}^n$.  
 \end{corollary}
 
 \begin{proof}
 The inequality follows by a direct application of Theorem \ref{thm:self-improving_general}. The only thing which remains is to prove a bound for the constant $C\left(\mu, \Psi\right)$. Observe that, as  the constant $C$ in the $A_\infty(\d\mu,\phi(L)(\d\mu))$ condition of $Y(Q):=\mu(Q)$ is clearly less than $c$, we have that 
\[C\left(\mu, \Psi\right)\leq\inf_{L>\max\{1,\Psi(c^{-1})^{-1}\}}      c_\mu 2^{n_\mu}\frac{L }{ 1-c {\Psi}^{-1}\left(\frac{1}{L}\right) }.\]
 Since by the preceding lemma in this case we have $\Psi^{-1}(t)=1/\phi^{-1}(1/t)$, then we can write
  \[C\left(\mu, \Psi\right)=\inf_{L>\max\{1,\Psi(c^{-1})^{-1}\}}      c_\mu 2^{n_\mu}\frac{L\phi^{-1}(L) }{ {\phi}^{-1}\left(L\right)  -c} .\]
 
 It is a simple Real Analysis exercise to find that the smallest value for the above function of $L$ is attained at the smallest $L>\max\{1,\Psi(c^{-1})^{-1}\}$ satisfying the identity
 \[
 L=\phi\left[c+c\frac{L[\phi^{-1}]'(L)}{\phi^{-1}(L)}\right].
 \]
 Observe that such an $L$ exists always because $\phi\left[c+c\frac{L[\phi^{-1}]'(L)}{\phi^{-1}(L)}\right]$ is a bounded function of $L$. Indeed, observe that, as $\phi$ is an increasing function, we can make the change of variables $L=\phi(s)$ to get 
 \[
 \phi\left[c+c\frac{L[\phi^{-1}]'(L)}{\phi^{-1}(L)}\right]=\phi\left[c+c\frac{\phi(s)}{s\phi'(s)}\right]\leq \phi\left[c+c\frac{1}{[\phi]_1}\right]
 \]
The existence is established then by checking that $\phi\left[c+c\frac{L[\phi^{-1}]'(L)}{\phi^{-1}(L)}\right]$ is greater than $1$ or greater than $\Psi(c^{-1})^{-1}$, depending on whether $\max\{1,\Psi(c^{-1})^{-1}\}$ is one quantity or the other.   If $\max\{1,\Psi(c^{-1})^{-1}\}=1$ then it happens that $c=1$ (note that $c$ is not allowed to be below $1$ by condition $\phi(1)=1$) and  then \[
 \phi\left[c+c\frac{L[\phi^{-1}]'(L)}{\phi^{-1}(L)}\right]>1 \iff    \frac{L[\phi^{-1}]'(L)}{\phi^{-1}(L)} >0,
 \]
 which trivially holds. In case $\max\{1,\Psi(c^{-1})^{-1}\}=\Psi(c^{-1})^{-1}$, one just has to check the existence of $L>\Psi(c^{-1})^{-1}$ such that
 \[
 L=\phi\left[c+c\frac{L[\phi^{-1}]'(L)}{\phi^{-1}(L)}\right],
 \]
 but note that $L>\Psi(c^{-1})^{-1}\iff c^{-1}>\Psi^{-1}(L)$, which for our choice of $\Psi^{-1}$ reads   $\phi^{-1}(L)>c$. By the continuity properties of the function under consideration, the desired existence will be proved if 
 \[
 c+c\frac{L[\phi^{-1}]'(L)}{\phi^{-1}(L)}>c,
 \]
 and this holds trivially because $\phi^{-1}$ is a positive increasing function.

 Therefore, by calling $A_\phi$ the set of those $L >\max\{1,\Psi(c^{-1})^{-1}\}$ satisfying the condition $L=\phi\left[c+c\frac{L[\phi^{-1}]'(L)}{\phi^{-1}(L)}\right],$
 \[
 \begin{split}
 C\left(\mu, \Psi\right)&=\inf_{L\in A_\phi}c_\mu 2^{n_\mu} \phi\left[c\left(1+ \frac{L[\phi^{-1}]'(L)}{\phi^{-1}(L)}\right)\right] \left( \frac{\phi^{-1}(L)}{L[\phi^{-1}]'(L)}+1\right)\\
 &\leq \inf_{\phi(s)\in A_\phi}c_\mu 2^{n_\mu} \phi\left[c\left(1+\frac{\phi(s)}{s\phi'(s)}\right)\right]\left(\frac{s\phi'(s)}{\phi(s)}+1\right),
 \end{split}
 \]  
 where the change of variables $L=\phi(s)$, $s>1$ has  been used again. 
 \end{proof}
 
 \begin{example}\label{ex:orlicz2}
 As an application  we compute the example \[
 {\phi}_{p,\alpha}(t)=t^{p}(1+\log^{+}(t))^{\alpha},\qquad \alpha>0, \quad p\geq1,
\]
which is  a submultiplicative Young function  defining the   Orlicz space $L^p(\log L)^\alpha$. First we note that, indeed, $ {\phi}_{p,\alpha}$ is submultiplicative, i.e.
\[
 {\phi}_{p,\alpha}(st)\leq {\phi}_{p,\alpha}(s) {\phi}_{p,\alpha}(t),\qquad s,t>0.
\]
We note that if $0<s<1$ and/or $0<t<1$ the inequality trivially holds. Hence we
shall assume that $s,t>1$. Note that then
\[
 {\phi}_{p,\alpha}(st)=s^{p}t^{p}(1+\log^+(st))^{\alpha}=s^{p}t^{p}(1+\log(st))^{\alpha}
\]
and it suffices to show that 
\[
1+\log(st)\leq(1+\log(s))(1+\log(t))
\]
but 
\begin{eqnarray*}
1+\log(st) & = & 1+\log(s)+\log(t)\\
& \leq & 1+\log(s)+\log(t)+\log(s)\log(t)\\
& = & (1+\log(s))(1+\log(t))
\end{eqnarray*}

and hence we are done.

Now observe that

\[
 {\phi}_{p,\alpha}'(t)=\begin{cases}
pt^{p-1} & \text{if}\quad t<1,\\
pt^{p-1}(1+\log(t))^{\alpha}+\alpha t^{p-1}(1+\log(t))^{\alpha-1} & \text{if}\quad t>1.
\end{cases}
\]
If $t>1$, then 
\[
 {\phi}_{p,\alpha}(t)=t^{p}(1+\log^{+}(t))^{\alpha}=t^{p}(1+\log(t))^{\alpha},
\]
and
\[
t {\phi}_{p,\alpha}'(t)=pt^{p}(1+\log(t))^{\alpha}+\alpha t^{p}(1+\log(t))^{\alpha-1}.
\]
Hence, 
\[
\frac{t {\phi}_{p,\alpha}'(t)}{ {\phi}_{p,\alpha}(t)}=p+\frac{\alpha}{1+\log(t)},\qquad t>1,
\]
and we then  have that
\[
p\leq\frac{t {\phi}_{p,\alpha}'(t)}{ {\phi}_{p,\alpha}(t)}\leq p+\alpha.
\]
These bounds are optimal for $t>1$, and so
we have that $[ {\phi}_{p,\alpha}]_{1}=p$ and $[ {\phi}_{p,\alpha}]_{2}=p+\alpha$.
By Corollary \ref{cor:OrliczBMO}, 
\[
\|f-f_{Q,\mu}\|_{ {\phi}_{p,\alpha}(L)\left(Q,\frac{\d\mu}{\mu(Q)}\right)}\leq c_{\mu}2^{n_{\mu}} {\phi}_{p,\alpha}\left(1+\frac{1}{[ {\phi}_{p,\alpha}]_{1}}\right)\left([ {\phi}_{p,\alpha}]_{2}+1\right)\|f\|_{\mathrm{BMO}(\d\mu)},
\]
and observe that, in this particular case,
\[
 {\phi}_{p,\alpha}\left(1+\frac{1}{[ {\phi}_{p,\alpha}]_{1}}\right)\left([ {\phi}_{p,\alpha}]_{2}+1\right)= {\phi}_{p,\alpha}\left(1+\frac{1}{p}\right)\left(p+\alpha+1\right),
\]
and 

\[
 {\phi}_{p,\alpha}\left(1+\frac{1}{p}\right)=\left(1+\frac{1}{p}\right)^{p}\left(1+\log\left(1+\frac{1}{p}\right)\right)^{\alpha}\leq e2^{\alpha}.
\]
Consequently,
\[
\|f-f_{Q,\mu}\|_{ {\phi}_{p,\alpha}(L)\left(Q,\frac{\d\mu}{\mu(Q)}\right)}\leq c_{\mu}2^{n_{\mu}}e2^{\alpha}\left(p+\alpha+1\right)\|f\|_{\mathrm{BMO}(\d\mu)}.
\]
This proves Corollary \ref{cor:LlogL} in the Introduction.

\end{example}
\begin{remark}
We observe that different choices for defining the same Orlicz norm may give different quantitative controls when applying our self-improving result. Indeed, one may check that, for instance, the alternative choice $\tilde{\phi}_{p,\alpha}(t):=[\log(e+1)]^{-\alpha}t^p[\log(e+t)]^\alpha$, $\alpha\geq0$, $p>1$ for defining the norm $\|\cdot\|_{L^p\log^\alpha L}$ leads to the following estimate 
  \[
  \begin{split}
 \|f-f_{Q,\mu}&\|_{L^p(\log L)^\alpha\left(Q,\frac{\d\mu}{\mu(Q)}\right)}\\
 &\leq c_\mu 2^{n_\mu}  e \left[\log\left(e+1\right)\right]^{\alpha(p-1)}[\log(e+2\log(1+e)^\alpha)]^\alpha   (p+\alpha+1) \|f\|_{\mathrm{BMO}(\d\mu)}.
\end{split}
 \] 
 This difference comes mainly from the fact that the Young function $\tilde{\phi}_{p,\alpha}$ is not submultiplicative but quasi-submultiplicative. Observe that the Young function we chose in the example above gives a cleaner constant. This difference makes us wonder about the sharpness of the estimates we get with our method. Nevertheless, observe that, in any case (that is, by choosing $\phi_{p,\alpha}$ or $\tilde{\phi}_{p,\alpha}$) we recover the sharp estimate in Theorem \ref{thm:PerezRela} by choosing $\alpha=0$.
\end{remark}

{\begin{remark} For a Young function $\phi$ we can define the weak Orlicz quasi-norm
\[
\|f\|_{\mathcal{M}_{\phi}(\d\mu)}:=\inf\left\{ \lambda>0\,:\,\sup_{t>0}\phi(t)\mu\left(\left\{ x\in\mathbb{R}^{n}\,:\,|f(x)|>\lambda t\right\} \right)\leq1\right\}. 
\]
It is easy to prove that 
\[
\|cf\|_{\mathcal{M}_{\phi}(\d\mu)}=|c|\|f\|_{\mathcal{M}_{\phi}(\d\mu)}
\]
for any $c\in\mathbb{R}$ 
and also that 
\[
\|f+g\|_{\mathcal{M}_{\phi}(\d\mu)}\leq2\left(\|f\|_{\mathcal{M}_{\phi}(\d\mu)}+\|g\|_{\mathcal{M}_{\phi}(\d\mu)}\right).
\]
This kind of spaces were studied in \cite{Iaffei}. This quasi-norm can be localized by defining 
\[
\left\Vert f\right\Vert _{\mathcal{M}_{\phi}\left(Q,\frac{\d\mu}{\mu(Q)}\right)}:=\inf\left\{ \lambda>0\,:\,\sup_{t>0}\frac{\phi(t)}{\mu(Q)}\mu\left(\left\{ x\in Q\,:\,|f(x)|>\lambda t\right\} \right)\leq1\right\} 
\]
for every cube $Q$ in $\mathbb{R}^n$. 

Such a quasi-norm can be proved to be good with respect to the functional $Y(Q):=\mu(Q)$ and it is possible to show that, counterparts of Lemmas \ref{lem:reduction}  and \ref{lem:OrliczGoodFunct} hold as well for weak-Orlicz spaces. Hence our approach allows   to provide results for this family of spaces as well.
\end{remark}
}

 \subsection{\texorpdfstring{$\mathrm{BMO}$}{BMO-variable}-type improvement at the variable Lebesgue spaces scale} \label{sec:variable-Lp}
 
We finish this section with another example of application now to the setting of variable Lebesgue spaces. Note that in this case the method of the Laplace transform for the obtention of the sharp estimate does not apply. Let $p:\mathbb{R}^n\to[1,\infty]$ be a Lebesgue measurable function and denote $p^-:=\essinf_{x\in\mathbb{R}^n}p(x)$ and $p^+:=\esssup_{x\in\mathbb{R}^n}p(x)$. Assume $p^+<\infty$. The Lebesgue space with variable exponent $p(\cdot)$ is the space of Lebesgue measurable functions $f$ satisfying that
\[
\|f\|_{L^{p(\cdot)}}:=\inf\left\{\lambda>0:\int_{\mathbb{R}^n} \left(\frac{|f(x)|}{\lambda}\right)^{p(x)}\,\d x\leq 1 \right\}<\infty.
\]
One can associate to this space the local averages 
\[
\|f\|_{L^{p(\cdot)}\left(Q,\frac{\d x}{|Q|}\right)} :=\inf\left\{\lambda>0:\frac{1}{|Q|}\int_{Q} \left(\frac{|f(x)|}{\lambda}\right)^{p(x)}\,\d x\leq 1 \right\}.
\]
 Note that, by choosing $\lambda_0=1$, one has 
\[
 \frac{1}{|Q|}\int_{Q} \left(\frac{|\chi_Q(x)|}{\lambda_0}\right)^{p(x)}\,\d x=\frac{1}{|Q|}\int_{Q} \left( |\chi_Q(x)| \right)^{p(x)}\,\d x=1,
\]
and therefore $\|\chi_Q\|_{L^{p(\cdot)}\left(Q,\frac{\d x}{|Q|}\right)}\leq 1$. 
This in particular means that the variable Lebesgue space $L^{p(\cdot)}(\mathbb{R}^n,\d x)$ satisfies properties \eqref{property:consistency} and \eqref{property:general_average} in Definition \ref{def:compatible}. Property \eqref{property:general_Fatou}  follows from \cite[Theorem 2.59]{CruzUribe2013}.
\begin{example}
We can show here that the functional defined by the Lebesgue measure satisfies the $A_\infty(\d\mu, L^{ p(\cdot)}(\d x))$ condition in Definition \ref{def:generalized_Ainfty2} for any   essentially bounded exponent function $p$. Indeed, let $Q$ be a cube in $\mathbb{R}^n$ and consider a family $\{Q_j\}_{j\in\mathbb{N}}\in \Delta(Q)$ and a sequence $\{h_j\}_{j\in\mathbb{N}}$ of functions satisfying $\|h_j\|_{L^{p(\cdot)}\left(Q_j,\frac{\d x}{|Q_j|}\right)}=1$ for every $j\in\mathbb{N}$. Then
\[
\begin{split}
\frac{1}{|Q|}\int_Q \left(\frac{\sum_{j\in\mathbb{N}}  h_j(x)\chi_{Q_j}(x)}{\lambda}\right)^{p(x)}\,\d x&= \sum_{j\in\mathbb{N}} \frac{|Q_j|}{|Q|}\frac{1}{|Q_j|}\int_{Q_j} \left(\frac{h_j(x)}{\lambda}\right)^{ p(x)}\,\d x
\end{split}
\]
and so, by taking  $\lambda= \left(\frac{\sum_{j\in\mathbb{N}}|Q_j|}{|Q|}\right)^{1/ p^+}$, one finds that 
\[
\begin{split}
\frac{1}{|Q|}\int_Q \left(\frac{\sum_{j\in\mathbb{N}}  h_j(x)\chi_{Q_j}(x)}{\lambda}\right)^{ p(x)}\,\d x&\leq  \sum_{j\in\mathbb{N}} \frac{|Q_j|}{|Q|}\frac{1}{|Q_j|}\int_{Q_j} \left(\frac{h_j(x)}{\left(\frac{\sum_{j\in\mathbb{N}}|Q_j|}{|Q|}\right)^{1/ p^+}}\right)^{ p(x)}\,\d x\\
& \leq \sum_{j\in\mathbb{N}}    \frac{|Q_j|}{|Q|}    \frac{1}{ \frac{\sum_{j\in\mathbb{N}}|Q_j|}{|Q|} }   \frac{1}{|Q_j|}\int_{Q_j}   h_j(x)^{p(x)}  \,\d x\\
&\leq 1,
\end{split}
\]
where \cite[Proposition 2.21]{CruzUribe2013} has been used. This proves that, for any $r\geq1$, 

\begin{equation} 
\left\|\sum_{j\in\mathbb{N}}  h_j\chi_{Q_j}\right\|_{L^{rp(\cdot)}\left(Q,\frac{\d x}{|Q|}\right)} \leq \left(\frac{\sum_{j\in\mathbb{N}}|Q_j|}{|Q|}\right)^{1/rp^+}.
\end{equation}
\end{example}

An application of this along with Theorem \ref{thm:self-improving_general} proves Corollary \ref{cor:variableLebesgue} in the Introduction. 
Now we will set an application of this corollary to a John-Nirenberg type inequality. Note first that, for given $1/p^-\leq s<\infty$, one has that $\||f|^s\|_{L^{p(\cdot)}\left(Q,\frac{\d x}{|Q|}\right)}=\|f\|_{L^{sp(\cdot)}\left(Q,\frac{\d x}{|Q|}\right)}^s$, see \cite[Proposition 2.18]{CruzUribe2013}. Let $t>0$ and take $r\geq 1$.  Define, for any cube $Q$ in $\mathbb{R}^n$, the subset $E\subset Q$ defined as $E_t:=\{x\in Q: |f(x)-f_Q|\geq t \}$. Then,

\begin{eqnarray*}
\frac{1}{|Q|}\int_Q \left(\frac{\chi_{E_t}(x) }{\frac{1}{t^{r }}\|f-f_Q\|^{r }_{L^{ rp(\cdot)}\left(Q,\frac{\d x}{|Q|}\right)}}\right)^{p(x)}\,\d x & = &
\frac{1}{|Q|}\int_{E_t} \left(\frac{t^r }{\||f-f_Q|^r\|_{L^{p(\cdot)}\left(Q,\frac{\d x}{|Q|}\right)}  }\right)^{p(x)}\,\d x\\
& \leq  &\frac{1}{|Q|} \int_{Q} \left(\frac{|f(x)-f_Q|^r }{\||f-f_Q|^r\|_{L^{ p(\cdot)}\left(Q,\frac{\d x}{|Q|}\right)}  }\right)^{p(x)}\,\d x.
\end{eqnarray*}

Hence, for any $t>0$, one has the following Chebychev type inequality
\[
\|\chi_{\{x\in Q: |f(x)-f_Q|\geq t\}}\|_{L^{p(\cdot)}\left(Q,\frac{\d x}{|Q|}\right)}\leq \frac{1}{t^r}\|f-f_Q\|_{L^{ rp(\cdot)}\left(Q,\frac{\d x}{|Q|}\right)}^r. 
\]
Then, by Corollary \ref{cor:variableLebesgue} applied to the exponent function $rp(\cdot)$,
\[
\|\chi_{\{x\in Q: |f(x)-f_Q|\geq t\}}\|_{L^{p(\cdot)}\left(Q,\frac{\d x}{|Q|}\right)}\leq  \frac{1}{t^r}\left[C(n) rp^+\|f\|_{\mathrm{BMO}} \right]^r, 
\]
for every $r\geq1$.
For $t\geq 2 C(n)p^+\|f\|_{\mathrm{BMO}}$, take $r=t/(2 C(n)p^+\|f\|_{\mathrm{BMO}})$ to  find that
\[
\|\chi_{\{x\in Q: |f(x)-f_Q|\geq t\}}\|_{L^{p(\cdot)}\left(Q,\frac{\d x}{|Q|}\right)}\leq   1/2^r=e^{-C(n,p^+)t/\|f\|_{\mathrm{BMO}}},
\]
where $C(n,p^+)=(2C(n)p^+)^{-1}\log 2$. When $t\leq 2C(n)p^+\|f\|_{\mathrm{BMO}}$, we have that the inequality  $e^{-C(n,p^+)t/\|f\|_{\mathrm{BMO}}}\geq 1/2$ holds and, therefore, by the Average property of the norm,
\[
\|\chi_{\{x\in Q: |f(x)-f_Q|\geq t\}}\|_{L^{p(\cdot)}\left(Q,\frac{\d x}{|Q|}\right)}\leq \|\chi_{Q}\|_{L^{p(\cdot)}\left(Q,\frac{\d x}{|Q|}\right)}\leq 1\leq  2e^{-C(n,p^+)t/\|f\|_{\mathrm{BMO}}}
\]
for these values of $t$. Hence, for every $t>0$ we got that 
\[
\|\chi_{\{x\in Q: |f(x)-f_Q|\geq t\}}\|_{L^{p(\cdot)}\left(Q,\frac{\d x}{|Q|}\right)}\leq  2 e^{-C(n,p^+)t/\|f\|_{\mathrm{BMO}}}.
\]
This proves Corollary \ref{cor:JohnNirenbergvariable}. The John-Nirenberg type inequality we just got has something to do  with the John-Nirenberg type inequality in \cite[Theorem 3.2]{Ho2014}. Note that, although condition $p^+<\infty$ is used, no further condition is assumed on the exponent function $p$. This inequality also proves and generalizes the inequality 
\[
\frac{1}{|Q|}\int_Q e^{\lambda|f(x)-f_Q|}\,\d x\leq C
\]
for $\lambda< C(n,p^+)$, as it does the one in \cite[Theorem 3.2]{Ho2014}.

\appendix
\section{Proof of Theorem \ref{pr:characterization_Ainfty}}\label{proof:characterization_Ainfty}
To provide the proof of Theorem \ref{pr:characterization_Ainfty} we need the following Lemma.
 \begin{lemma}\label{lem:Hyt}
 Let $b\in L_{\text{loc}}^{1}(\mathbb{R}^{n},\d\mu)$ and let $Q_{0}\subset\mathbb{R}^{n}$
be a cube. Then there exists a collection $\mathcal{S}(Q_{0})\subset\mathcal{D}(Q_{0})$ of dyadic cubes such that:
\begin{enumerate}
\item There exists a family $\{E_Q\}_{Q\in \mathcal{S}(Q_{0})}$ of pairwise disjoint sets with $E_Q\subset Q$ and $\mu(Q)\leq 2 \mu (E_Q)$ for each $Q\in \mathcal{S}(Q_{0})$.
\item For almost every $x\in Q_{0}$, 
\[
|f(x)-f_{Q_{0}}|\chi_{Q_{0}}(x)\lesssim\sum_{Q\in\mathcal{S}(Q_{0})}\frac{1}{\mu(Q)}\int_{Q}|f-f_{Q}|\d\mu\chi_{Q}(x).
\]
\end{enumerate}
\end{lemma}
We omit the proof since the result follows straightforward adapting arguments in \cite{Hyt}.  

Before proceeding to the proof of Theorem \ref{pr:characterization_Ainfty} we need to settle another lemma. It turns out that the proof of the theorem in the Euclidean case with Lebesgue measure uses the fact that, given any cube $Q$, there is always a subcube of it with half its measure. This is not guaranteed in principle for a doubling measure (at least it is not for the best of our knowledge). Therefore we will provide a geometric lemma which is enough for our purposes. 
\begin{lemma}\label{Cubes:claim}
Let $\mu$ be a doubling measure on $\mathbb{R}^n$ with doubling dimension $n_\mu$ and doubling constant $c_\mu$. If $\mu$ is not identically zero, then there is, for any cube $Q$ in $\mathbb{R}^n$, a subcube $\tilde{Q}$ with $\mu(\tilde{Q})=\alpha \mu(Q)$, where $\frac{1}{4c_\mu}\leq \min\{\alpha, 1-\alpha\}$.
\end{lemma}

\begin{proof}
Note that a nontrivial doubling measure must satisfy that $\mu(Q)>0$ for every cube $Q$ in $\mathbb{R}^n$. Indeed, assume that there exists a cube $Q$  with $\mu(Q)=0$. Since we can write $\mathbb{R}^n=\bigcup_{k\in\mathbb{N}}kQ$, we would have that
\[
\mu(\mathbb{R}^n)=\mu\left(\bigcup_{k\in\mathbb{N}}kQ\right) =\lim_{k\to\infty}\mu(kQ)\leq\lim_{k\to\infty} c_\mu k^{n_\mu}\mu(Q)=\lim_{k\to\infty}c_\mu k^{n_\mu}\cdot 0=0,
\]
which contradicts the nontriviality of $\mu$.

For any given $x\in\mathbb{R}^n$ and $t\geq0$ let us denote by $Q(x,t)$ the open cube with center at $x$ and sidelength $t$. Fix a cube $Q$ in $\mathbb{R}^n$ and let $x_Q$ be its center. Taking into account the preceding observation, we have that  $\mu(Q(x_Q,t))>0$ for every $0<t\leq\ell(Q)$. Moreover, since one can fit a cube $P$ inside any annulus $Q(x_Q,\ell(Q))\backslash Q(x_Q,t)$ with $0<t<\ell(Q)$, we do know that also $\mu \left[Q(x_Q,\ell(Q))\backslash Q(x_Q,t)\right]>0$.  This implies that the function $h:[0,\ell(Q)]\to[0,\infty)$ defined by $h(t)=\mu[Q(x_Q,t)]$  is strictly increasing. Note that, as  $Q(x_Q,t)=\bigcup_{0<s<t}Q(x_Q,s)$, we always have that 
\begin{equation}\label{inf}
\lim_{\varepsilon\to 0} h(t)-h(t-\varepsilon) =0,
\end{equation}
and, therefore, the only possibility for a discontinuity of $h$ at a point $t$ is to have 
\begin{equation}\label{sup}
\lim_{\varepsilon\to 0} h(t+\varepsilon)-h(t)>0,
\end{equation}
 that is, to have 
 \[
0<\lim_{\varepsilon\to 0} \mu[Q(x_Q,t+\varepsilon)]-\mu[Q(x_Q,t)]=\mu[\overline{Q(x_Q,t)}]-\mu[Q(x_Q,t)]=\mu[\partial Q(x_Q,t)],
\] 
where it has been used that the closure $\overline{Q(x_Q,t)}$ of the cube $Q(x_Q,t)$ can be written as the intersection $\bigcap_{t<s\leq \ell(Q) }Q(x_Q,s)$.

In case such a discontinuity happens, note that, by the doubling condition, 
\[
\begin{split}
\mu[\overline{Q(x_Q,t)}]&\leq c_\mu \mu[Q(x_Q,t)]=c_\mu\left[\mu[\overline{Q(x_Q,t)}]-\mu[\partial Q(x_Q,t)]\right],
\end{split}
\]
and so we have  $\mu[\partial Q(x_Q,t)] \leq \frac{c_\mu-1}{c_\mu}\mu[\overline{Q(x_Q,t)}]$. We can uniformly bound this obtaining that 
\begin{equation}\label{uniformbound}
\mu[\partial Q(x_Q,t)]\leq \frac{c_\mu-1}{c_\mu}\mu(Q),\qquad 0<t<\ell(Q).
\end{equation}
Therefore, $h$ must be continuous except for jumps of length at most $ \frac{c_\mu-1}{c_\mu}\mu(Q)$. These jumps are gaps of  $h\left([0,\ell(Q)]\right)$ in $  [0,\mu(Q)]$.  Let  $G$ be this set of gaps of $h\left([0,\ell(Q)]\right)$ in $[0,\mu(Q)]$, namely,
\[
G:=[0,\mu(Q)]\backslash h\left([0,\ell(Q)]\right).
\]
Since $h$ is strictly increasing, we  know that $G$ is at most the countable union of its connected components and moreover, we know that there are points of $h\left([0,\ell(Q)]\right)$ in $[0,\mu(Q)]$ between any two connected components of $G$. The goal is to see that there is always a connected component $I$ of $G$ for which one can find points in $h\left([0,\ell(Q)]\right)$  which are close to $I$ and far from the boundary of $[0,\mu(Q)]$, that is, we look for some $\alpha\in(0,1)$ with $\alpha\mu(Q)$ close to $I$ and $\min\{\alpha,1-\alpha\}$ uniformly bounded from below.

We investigate the following  two possibilities:
\begin{enumerate}
\item There is $t\in (0,\ell(Q))$ with $h(t)=\mu(Q(x_Q,t))=\frac{1}{2}\mu(Q )$. In this case, we can choose $\alpha=1/2$ and we are done.
\item There is \emph{not} any  $t\in (0,\ell(Q))$ with $h(t)=\mu(Q(x_Q,t))=\frac{1}{2}\mu(Q)$, that is, $\frac{1}{2}\mu(Q)\in G$. Let us call $I$ the connected component of $G$ containing $\frac{1}{2}\mu(Q)$. Note that $I$ must be a nondegenerated interval  containing $\frac{1}{2}\mu(Q)$ since the discontinuities of $h$ are jump discontinuities. Moreover, by \eqref{inf} and \eqref{sup} we have that $I$ must be of the form $(\inf I,\sup I]$.  Furthermore,  in virtue of the bound in \eqref{uniformbound}, its length can be at most $\frac{c_\mu-1}{c_\mu}\mu(Q)$. Around this interval $I$ we can find points of $h\left([0,\ell(Q)]\right)$. We will choose one of these depending on the closeness of $I$ to the borders of $[0,\mu(Q)]$. Assume for instance that $\sup I$ is closer to $\mu(Q)$ than $\inf I$ is to $0$. In this case, 
\[
\begin{split}
\inf I&=\inf I-0\geq \mu(Q)-\sup I\geq \mu(Q)-(\inf I+|I|)\\
&\geq \mu(Q)-\inf I-\frac{c_\mu-1}{c\mu}\mu(Q)=\frac{1}{c_\mu}\mu(Q)-\inf I,
\end{split}
\]
which implies that $\inf I\geq \frac{1}{2c_\mu}\mu(Q)$. Then we can choose any $\alpha$ with $\alpha \mu(Q)\in \left(\inf I-\frac{1}{4 c_\mu} {\mu(Q)},\inf I  \right]\cap h\left([0,\ell(Q)]\right)\neq \emptyset$.

 Since $\inf I\geq \frac{1}{2c_\mu}\mu(Q)$, we know that 
\[
\inf I-\frac{1}{4c_\mu} {\mu(Q)}\geq  \frac{1}{2c_\mu} {\mu(Q)}-\frac{1}{4c_\mu} {\mu(Q)}=\frac{1}{4c_\mu} {\mu(Q)},
\]
so $\alpha\geq \frac{1}{4c_\mu}$. Since $\alpha \mu(Q)<\inf I\leq \frac{1}{2}\mu(Q)$, we know that $\mu(Q)-\alpha\mu(Q) \geq \frac{1}{2}\mu(Q)\geq \frac{1}{4c_\mu}\mu(Q)$, where we used that $c_\mu\geq 1$. This proves that also $1-\alpha\geq \frac{1}{4c_\mu}$.

{Note that we could actually have chosen $\alpha$ such that $\alpha\mu(Q)=\inf I$, since $\inf I\notin I$, but we chose to write the argument in this form in order to have a valid argument also for the case   in which $\inf I$ is closer to $0$ than $\sup I$ is to $\mu(Q)$. Although in this other case we do not have $\sup I\notin I$, we can still apply the argument  above  to find an appropriate $\alpha$ and then we are done. }
\end{enumerate}
\end{proof}

With these tools at hand, we are prepared for the proof of Theorem \ref{pr:characterization_Ainfty}.

\begin{proof}[Proof of Theorem \ref{pr:characterization_Ainfty}]
First we observe that by Lemma \ref{lem:Hyt}
\begin{align*}
\frac{1}{Y(Q_{0})}\int_{Q_{0}}|f(x)-f_{Q_{0}}|v(x)\d\mu(x)  &\lesssim\frac{1}{Y(Q_{0})}\sum_{Q\in\mathcal{S}(Q_{0})}\frac{1}{\mu(Q)}\int_{Q}|f-f_{Q}|\d\mu v(Q)\\
 & \leq\|b\|_{BMO(\d\mu)}\frac{1}{Y(Q_{0})}\sum_{Q\in\mathcal{S}(Q_{0})}v(Q)\\
 &\leq\|b\|_{BMO(\d\mu)}\frac{1}{Y(Q_{0})}\sum_{Q\in\mathcal{S}(Q_{0})}\frac{\mu(Q)}{\mu(Q)}v(Q)\\
 & \leq2\|b\|_{BMO(\d\mu)}\frac{1}{Y(Q_{0})}\sum_{Q\in\mathcal{S}(Q_{0})}\frac{\mu(E_{Q})}{\mu(Q)}v(Q) \\
 &\leq2\|b\|_{BMO(\d\mu)}\frac{1}{Y(Q_{0})}\sum_{Q\in\mathcal{S}(Q_{0})}\int_{E_{Q}}M(\chi_{Q_{0}}v)\d\mu\\
 & \leq2\|b\|_{BMO(\d\mu)}\frac{1}{Y(Q_{0})}\int_{Q_{0}}M(\chi_{Q_{0}}v)\d\mu\\
 & \leq2\|b\|_{BMO(\d\mu)}[v]_{A_{\infty,Y}}
\end{align*}
This yields $B\lesssim [v]_{A_{\infty,Y}}$.

 For the other implication, recall that, for any weight $w\in L^1_{\loc}(\mathbb{R}^n,\d\mu)$ it holds that
\begin{equation}\label{eq:maximal_log}
\frac{1}{\mu(Q)}\int_Q M_\mu(w\chi_Q)(x)\,\d\mu(x) \lesssim \frac{1}{\mu(Q)}\int_Q\left[1+\log^+\left(\frac{w(x)\chi_Q(x)}{w_{Q,\mu}}\right)\right]w(x)\,\d\mu(x), 
\end{equation}
where $\log^+ t:=\max\{\log t,0\}$. Assume that there is some constant $C>0$ such that
 \begin{equation} \label{eq:BMOAinfty}
 \frac{1}{w(Q)}\int_Q|f(x)-f_{Q,\mu}| \,\d w(x) \leq B  \| f\|_{\mathrm{BMO}(\d\mu)}
 \end{equation}
 for every function $f\in \mathrm{BMO}(\d\mu)$. The first observation is the fact that $w(Q)\leq 8c_\mu B\cdot Y(Q)$ for every cube $Q$ in $\mathbb{R}^n$.  Indeed, by Lemma \ref{Cubes:claim}, given a cube $Q$ in $\mathbb{R}^n$,  there is a cube $\tilde{Q}\subset Q$  of measure $\mu(\tilde{Q})=\alpha\mu(Q)$ with $\min\{\alpha,1-\alpha\}\geq \frac{1}{4c_\mu}$. For these cubes we have
 
 \[
| \chi_Q(x)-(\chi_{\tilde{Q}})_{Q,\mu}|\geq \min\{\alpha,1-\alpha\}\geq \frac{1}{4c_\mu}.
 \] 
Since $\chi_{\tilde{Q}}\in L^\infty(\mathbb{R}^n,\d\mu)\subset \mathrm{BMO}(\d\mu)$, one can write
\[
\begin{split}
 \min\{\alpha,1-\alpha\}\frac{w(Q)}{Y(Q)}&\leq \frac{1}{Y(Q)}\int_Q|\chi_{\tilde{Q}}(x)-(\chi_{\tilde{Q}})_{Q,\mu}|\,\d w(x)\\
 &\leq \|\chi_{\tilde{Q}}\|_{\mathrm{BMO}_{w\d\mu,Y}}\leq B \|\chi_{\tilde{Q}}\|_{\mathrm{BMO}(\d\mu)}\leq 2B,
\end{split} 
\]
which finishes the proof of the claimed inequality $w(Q)<8c_\mu BY(Q)$.

Define the weight $v(x):= M_\mu\left(\frac{w\chi_Q}{w_{Q,\mu}}\right)(x)^{1/2}$, which  satisfy that $f(x):=\log v(x)$ is a $\mathrm{BMO}(\d\mu)$ function with $\|f\|_{\mathrm{BMO}(\d\mu)}\leq  4[v]_{A_1(\d\mu)}^2\leq c(\mu)$, where $c(\mu)$ is a constant which just depend on the doubling dimension of $\mu$, but not on the cube $Q$, although $v$ does depend on it.  This means that, for this function $f$,
\[
\frac{1}{Y(Q)}\int_Q|f(x)-f_{Q,\mu}|\,\d w(x) \leq B\|f\|_{\mathrm{BMO}(\d\mu)}\leq  B \cdot c(\mu),
\]
that is,
\begin{equation}\label{eq:stepBMOAinfty}
\frac{1}{\mu(Q)}\int_Q|f(x)-f_{Q,\mu}|\,\d w(x)  \leq  B\cdot c(\mu)\frac{Y(Q)}{\mu(Q)}.
\end{equation}
There exists some $\beta_\mu>1$ such that, for any $x\in L_Q:=\{x\in Q:w(x)\geq \beta_\mu^2 w_{Q,\mu}\}$,
\begin{equation}\label{eq:claimBMOAinfty}
|f(x)-f_{Q,\mu}|\geq \frac{1}{2}\log^+\left(\frac{ w(x)\chi_Q(x)}{\beta_\mu^2 w_{Q,\mu}}\right).
\end{equation}
Indeed, note first that, by Jensen's inequality and Kolmogorov's inequality (see \cite[Lemma 5.16]{Duoandikoetxea2001}), 
\[
\begin{split}
f_{Q,\mu}&=\frac{1}{\mu(Q)}\int_Q \log v(x)\,\d\mu(x)=\frac{1}{\mu(Q)}\int_Q\log\left(\frac{M_\mu(w\chi_Q)(x)}{w_{Q,\mu}}\right)^{1/2}\,\d\mu(x) \\
&\leq \log\left[\frac{1}{\mu(Q)}\int_Q \left(\frac{M_\mu(w\chi_Q)(x)}{w_{Q,\mu}}\right)^{1/2}\,\d\mu(x) \right]\leq \log\left(2\|M_\mu\|_{L^1\to L^{1,\infty}}^{1/2}\right).
\end{split}
\]
Then for any $x\in L_Q=\{x\in Q:w(x)\geq \beta_\mu^2 w_{Q,\mu}\}$,  
\[
\begin{split}
f_{Q,\mu}& \leq  \log\left(2\|M_\mu\|_{L^1\to L^{1,\infty}}^{1/2}\right)\leq \log\left(\frac{(w(x)\chi_Q(x))^{1/2}}{(w_{Q,\mu})^{1/2}}\right) \\
&\leq \log\left( \frac{M(w\chi_Q)(x)^{1/2}}{(w_{Q,\mu})^{1/2}}\right)=\log v(x)=b(x)
\end{split}
\]
if $\beta_\mu$ is chosen to be equal to $2\|M_\mu\|_{L^1\to L^{1,\infty}}^{1/2}$. With this choice then one has that, for any $x\in L_Q$, 
\[
\begin{split}
|f(x)-f_{Q,\mu}|&=f(x)-f_{Q,\mu}\geq f(x)-\log \beta_\mu=\log\left(\frac{v(x)}{\beta_\mu}\right)\\
&= \frac{1}{2}\log\left[\frac{M(w\chi_Q)(x)}{\beta_\mu^2 w_{Q,\mu}}\right]\geq \frac{1}{2}\log\left[\frac{w(x)\chi_Q(x)}{\beta_\mu^2 w_{Q,\mu}}\right].
\end{split}
\]
This proves the claimed inequality \eqref{eq:claimBMOAinfty}. Use now  inequality $w(Q)<8{c_\mu} BY(Q)$  and inequality \eqref{eq:stepBMOAinfty}    together with \eqref{eq:claimBMOAinfty} in \eqref{eq:maximal_log} to get
\[
\begin{split}
&\frac{1}{\mu(Q)}\int_Q M_\mu(w\chi_Q)(x)\,\d\mu(x) \leq C(\mu)\frac{1}{\mu(Q)}\int_Q\left[1+\log^+\left(\frac{w(x)\chi_Q(x)}{w_{Q,\mu}}\right)\right]w(x)\,\d\mu(x)\\
&\leq C(\mu)\frac{w(Q)}{\mu(Q)}(1+2\log\beta_\mu)+\frac{C(\mu)}{\mu(Q)}\int_Q\log^+\left(\frac{w(x)\chi_Q(x)}{\beta_\mu^2w_{Q,\mu}}\right)w(x)\,\d\mu(x)\\
&\leq C(\mu)\frac{Y(Q)}{\mu(Q)}\left[ 8{c_\mu}(1+2\log\beta_\mu)+ 2c(\mu)\right]B,
\end{split}
\]
or, equivalently 
\[
	\frac{1}{Y(Q)}\int_Q M_\mu(w\chi_Q)(x)\,\d\mu(x)  \leq  C(\mu)\left[ 8{c_\mu}(1+2\log\beta_\mu)+ 2c(\mu)\right]B.
\]
Since the above estimate is   independent of $Q$, it has been proved that $[w]_{A_{\infty,Y}(\d\mu)}<C(\mu)\left[ 8{c_\mu}(1+2\log\beta_\mu) + 2c(\mu)\right]B$, so the desired result follows.
\end{proof}

\section*{Acknowledgements}
J. C. M.-P. is supported by the Basque Government through the BERC 2018-2021 program and by Spanish Ministry of Science, Innovation and Universities through BCAM Severo Ochoa accreditation SEV-2017-0718. He is also supported by MINECO through the MTM2017-82160-C2-1-P project funded by (AEI/FEDER, UE), acronym ``HAQMEC'', through ''la Caixa'' grant, and through the MATHROCKS project, funded by European Commission with Grant Agreement number 777778 (H2020-MSCA-RISE-2017).   He is also grateful to the people of the Universidad de Buenos Aires and the Universidad Nacional del Sur for their hospitality during his visit to Argentina in 2019.

E.R. is partially supported by grants UBACyT 20020170200057BA and PIP (CONICET)
11220110101018. This project has received funding from the European
Union's Horizon 2020 research and innovation programme under the Marie Sklodowska-Curie grant agreement No 777822.

I.P.R.-R. is partially supported by grants PIP (CONICET) 11220130100329CO and PICT 2018-02501 (Agencia I+D+i).

\bibliographystyle{amsalpha}

\providecommand{\bysame}{\leavevmode\hbox to3em{\hrulefill}\thinspace}
\providecommand{\MR}{\relax\ifhmode\unskip\space\fi MR }
\providecommand{\MRhref}[2]{%
  \href{http://www.ams.org/mathscinet-getitem?mr=#1}{#2}
}
\providecommand{\href}[2]{#2}


\end{document}